\documentclass[10pt]{amsart}
\UseRawInputEncoding
\usepackage{amsfonts}
\usepackage{mathrsfs}
\usepackage{amscd}
\usepackage{amsmath}
\usepackage{amssymb}
\usepackage{latexsym}
\usepackage{lscape}
\usepackage{xypic}
\usepackage{comment}
\usepackage{amscd}
\usepackage{wasysym}  
\usepackage{tikz} 
\usetikzlibrary{matrix,arrows}
\usepackage{appendix}
\usepackage{geometry}
\usepackage[nice]{nicefrac}
\usepackage{dsfont}
\usepackage[utf8]{inputenc}  
\usepackage{xypic}

\usepackage[pagebackref,hyperindex,linktocpage=true]{hyperref}
\hypersetup{
    colorlinks,
    linkcolor={red!50!black},
    citecolor={blue!50!black},
    urlcolor={blue!80!black}
}

\usepackage{color}
\newcommand{\DrFibDim}[4]{\mathbf{d}_{\mathrm{Dr}}\!\left(#1,#2\,;\,#3,#4\right)}
\DeclareMathOperator{\Fr}{Fr}

\newtheorem{thm}{Theorem}[section]

\newtheorem{lem}[thm]{Lemma}

\newtheorem{lem-def}[thm]{Lemma-Definition}
\newtheorem{cor}[thm]{Corollary}

\newtheorem{fact}[thm]{Fact}

\theoremstyle{definition}

\newtheorem{ex}[thm]{Example}

\newtheorem{rmk}[thm]{Remark}
\newtheorem{dfn}[thm]{Definition}

\numberwithin{equation}{section}

\newcommand{\nc}{\newcommand}

\nc{\on}{\operatorname}
\nc{\fraka}{{\mathfrak a}} \nc{\bba}{{\mathbf a}}
\nc{\frakb}{{\mathfrak b}}
\nc{\frakc}{{\mathfrak c}}
\nc{\frakd}{{\mathfrak d}}
\nc{\frake}{{\mathfrak e}}
\nc{\frakf}{{\mathfrak f}}
\nc{\frakg}{{\mathfrak g}}
\nc{\frakh}{{\mathfrak h}}
\nc{\fraki}{{\mathfrak i}}
\nc{\frakj}{{\mathfrak j}}
\nc{\frakk}{{\mathfrak k}}
\nc{\frakl}{{\mathfrak l}}
\nc{\frakm}{{\mathfrak m}}
\nc{\frakn}{{\mathfrak n}}
\nc{\frako}{{\mathfrak o}}
\nc{\frakp}{{\mathfrak p}}
\nc{\frakq}{{\mathfrak q}}
\nc{\frakr}{{\mathfrak r}}
\nc{\fraks}{{\mathfrak s}}
\nc{\frakt}{{\mathfrak t}}
\nc{\fraku}{{\mathfrak u}}
\nc{\frakv}{{\mathfrak v}}
\nc{\frakw}{{\mathfrak w}}
\nc{\frakx}{{\mathfrak x}}
\nc{\fraky}{{\mathfrak y}}
\nc{\frakz}{{\mathfrak z}}
\nc{\frakA}{{\mathfrak A}}
\nc{\frakB}{{\mathfrak B}}
\nc{\frakC}{{\mathfrak C}}
\nc{\frakD}{{\mathfrak D}}
\nc{\frakE}{{\mathfrak E}}
\nc{\frakF}{{\mathfrak F}}
\nc{\frakG}{{\mathfrak G}}
\nc{\frakH}{{\mathfrak H}}
\nc{\frakI}{{\mathfrak I}}
\nc{\frakJ}{{\mathfrak J}}
\nc{\frakK}{{\mathfrak K}}
\nc{\frakL}{{\mathfrak L}}
\nc{\frakM}{{\mathfrak M}}
\nc{\frakN}{{\mathfrak N}}
\nc{\frakO}{{\mathfrak O}}
\nc{\frakP}{{\mathfrak P}}
\nc{\frakQ}{{\mathfrak Q}}
\nc{\frakR}{{\mathfrak R}}
\nc{\frakS}{{\mathfrak S}}
\nc{\frakT}{{\mathfrak T}}
\nc{\frakU}{{\mathfrak U}}
\nc{\frakV}{{\mathfrak V}}
\nc{\frakW}{{\mathfrak W}}
\nc{\frakX}{{\mathfrak X}}
\nc{\frakY}{{\mathfrak Y}}
\nc{\frakZ}{{\mathfrak Z}}
\nc{\bbA}{{\mathbb A}}
\nc{\bbB}{{\mathbb B}}
\nc{\bbC}{{\mathbb C}}
\nc{\bbD}{{\mathbb D}}
\nc{\bbE}{{\mathbb E}}
\nc{\bbF}{{\mathbb F}} \nc{\bbf}{{\mathbf f}}
\nc{\bbG}{{\mathbb G}}
\nc{\bbH}{{\mathbb H}}
\nc{\bbI}{{\mathbb I}}
\nc{\bbJ}{{\mathbb J}}
\nc{\bbK}{{\mathbb K}}
\nc{\bbL}{{\mathbb L}}
\nc{\bbM}{{\mathbb M}}
\nc{\bbN}{{\mathbb N}}
\nc{\bbO}{{\mathbb O}}
\nc{\bbP}{{\mathbb P}}
\nc{\bbQ}{{\mathbb Q}}
\nc{\bbR}{{\mathbb R}}
\nc{\bbS}{{\mathbb S}}
\nc{\bbT}{{\mathbb T}}
\nc{\bbU}{{\mathbb U}}
\nc{\bbV}{{\mathbb V}}
\nc{\bbW}{{\mathbb W}}
\nc{\bbX}{{\mathbb X}}
\nc{\bbY}{{\mathbb Y}}
\nc{\bbZ}{{\mathbb Z}}
\nc{\calA}{{\mathcal A}}
\nc{\calB}{{\mathcal B}}
\nc{\calC}{{\mathcal C}}
\nc{\calD}{{\mathcal D}}
\nc{\calE}{{\mathcal E}}
\nc{\calF}{{\mathcal F}}
\nc{\calG}{{\mathcal G}}
\nc{\calH}{{\mathcal H}}
\nc{\calI}{{\mathcal I}}
\nc{\calJ}{{\mathcal J}}
\nc{\calK}{{\mathcal K}}
\nc{\calL}{{\mathcal L}}
\nc{\calM}{{\mathcal M}}
\nc{\calN}{{\mathcal N}}
\nc{\calO}{{\mathcal O}}
\nc{\calP}{{\mathcal P}}
\nc{\calQ}{{\mathcal Q}}
\nc{\calR}{{\mathcal R}}
\nc{\calS}{{\mathcal S}}
\nc{\calT}{{\mathcal T}}
\nc{\calU}{{\mathcal U}}
\nc{\calV}{{\mathcal V}}
\nc{\calW}{{\mathcal W}}
\nc{\calX}{{\mathcal X}}
\nc{\calY}{{\mathcal Y}}
\nc{\calZ}{{\mathcal Z}}

\nc{\scrA}{{\mathscr A}}
\nc{\scrB}{{\mathscr B}}
\nc{\scrR}{{\mathscr R}}

\nc{\Bmu}{\mbox{$\raisebox{-0.59ex}{$l$}\hspace{-0.18em}\mu\hspace{-0.88em}\raisebox{-0.98ex}{\scalebox{2}{$\color{white}.$}}\hspace{-0.416em}\raisebox{+0.88ex}{$\color{white}.$}\hspace{0.46em}$}{}}

\nc{\bnu}{{\bar{ \nu}}}

\nc{\olO}{\bar{\calO}}

\nc{\al}{{\alpha}} 
\nc{\be}{{\beta}}
\nc{\ga}{{\gamma}} \nc{\Ga}{{\Gamma}}
 \nc{\hGa}{\hat{\Gamma}}
\nc{\ve}{{\varepsilon}} 
\nc{\la}{{\lambda}} \nc{\La}{{\Lambda}}
\nc{\om}{\omega} \nc{\Om}{\Omega} 
\nc{\sig}{{\sigma}} \nc{\Sig}{{\Sigma}}

\nc{\tnb}{\psi_{\rm tame}}
\nc{\oM}{\overline{{M}}}
\nc{\op}{{\on{op}}}
\nc{\ad}{{\on{ad}}}
\nc{\alg}{{\on{alg}}}
\nc{\Ad}{{\on{Ad}}}
\nc{\Adm}{{\on{Adm}}} \nc{\aff}{{\on{aff}}}
\nc{\Aut}{{\on{Aut}}}
\nc{\Bun}{{\on{Bun}}}
\nc{\cha}{{\on{char}}}
\nc{\der}{{\on{der}}}
\nc{\Der}{{\on{Der}}}
\nc{\diag}{{\on{diag}}}
\nc{\End}{{\on{End}}}
\nc{\Fl}{{\calF\!\ell}}
\nc{\Tr}{{\on{Transp}}}
\nc{\TR}{{\calT\!\calR}}
\nc{\Gal}{{\on{Gal}}}
\nc{\Gr}{{\on{Gr}}}
\nc{\rH}{{\on{H}}}
\nc{\Hom}{{\on{Hom}}}
\nc{\IC}{{\on{IC}}}
\nc{\id}{{\on{id}}}
\nc{\Id}{{\on{Id}}}
\nc{\ind}{{\on{ind}}}
\nc{\Ind}{{\on{Ind}}}
\nc{\Lie}{{\on{Lie}}}
\nc{\Pic}{{\on{Pic}}}
\nc{\pr}{{\on{pr}}}
\nc{\Res}{{\on{Res}}}
\nc{\res}{{\on{res}}} \nc{\Sat}{{\on{Sat}}}
\nc{\s}{{\on{sc}}}
\nc{\drv}{{\on{der}}}
\nc{\sgn}{{\on{sgn}}}
\nc{\Spec}{{\on{Spec}}}\nc{\Spf}{\on{Spf}} 
\nc{\Sph}{\on{Sph}}
\nc{\St}{{\on{St}}}
\nc{\tr}{{\on{tr}}}
\nc{\Mod}{{\mathrm{-Mod}}}
\nc{\Hilb}{{\on{Hilb}}} 
\nc{\Ext}{{\on{Ext}}} 
\nc{\vs}{{\on{Vec}}}
\nc{\ev}{{\on{ev}}}
\nc{\nO}{{\breve{\calO}}}
\nc{\tS}{{\tilde{S}}}
\nc{\spe}{{\on{sp}}}
\nc{\loc}{{\on{loc}}}

\nc{\nscrR}{{\mathscr{R}^{\on{nr}}}}

\nc{\GL}{{\on{GL}}}
\nc{\U}{{\on{U}}}
\nc{\Gl}{\on{Gl}} 
\nc{\GSp}{{\on{GSp}}}
\nc{\gl}{{\frakg\frakl}}
\nc{\SL}{{\on{SL}}} 
\nc{\SU}{{\on{SU}}} 
\nc{\SO}{{\on{SO}}}
\nc{\PGL}{{\on{PGL}}}

\nc{\Conv}{{\on{Conv}}}
\nc{\Rep}{{\on{Rep}}}
\nc{\Dom}{{\on{Dom}}}
\nc{\red}{{\on{red}}}
\nc{\act}{{\on{act}}}
\nc{\nr}{{\on{nr}}}
\nc{\ctf}{{\on{ctf}}}

\nc{\str}{{\on{-}}} 
\nc{\os}{{\bar{s}}}
\nc{\oeta}{{\bar{\eta}}}

\nc{\hookto}{\hookrightarrow}
\nc{\longto}{\longrightarrow}
\nc{\leftto}{\leftarrow}
\nc{\onto}{\twoheadrightarrow}
\nc{\lonto}{\twoheadleftarrow}

\nc{\uG}{{\underline{G}}}
\nc{\uA}{{\underline{A}}}
\nc{\uS}{{\underline{S}}}
\nc{\uT}{{\underline{T}}}
\nc{\uM}{{\underline{M}}}
\nc{\uP}{{\underline{P}}}
\nc{\uB}{{\underline{B}}}
\nc{\uN}{{\underline{N}}}

\nc{\ucG}{{\underline{\calG}}}
\nc{\ucA}{{\underline{\calA}}}
\nc{\ucS}{{\underline{\calS}}}
\nc{\ucT}{{\underline{\calT}}}
\nc{\ucM}{{\underline{\calM}}}
\nc{\ucP}{{\underline{\calP}}}
\nc{\ucN}{{\underline{\calN}}}

\nc{\bF}{{\breve{F}}}

\nc{\oFl}{{\overline{\Fl}}} 
\nc{\bU}{{\overline{U}}}
\nc{\tGr}{{\tilde{\Gr}}}
\nc{\cGr}{\calG\! r}
\nc{\oGr}{\overline{\on{Gr}}} 
\nc{\ocGr}{\overline{\calG\! r}}
\nc{\co}{{\colon}}
\nc{\sch}[1]{(Sch/{#1})}
\nc{\HypLoc}[1]{HypLoc({#1})}

\nc{\ohtimes}{\stackrel{!}{\otimes}}
\nc{\boxtilde}{\widetilde{\boxtimes}}
\nc{\vstar}{{\varhexstar}}

\nc{\Div}{\on{Div}}

\nc{\bslash}{\backslash}
\nc{\algQl}{{\bar{\bbQ}_\ell}}
\nc{\sF}{{\bar{F}}}
\nc{\nF}{{\breve{F}}}
\nc{\nW}{{W^{\on{nr}}}}
\nc{\sk}{{\bar{k}}}
\nc{\cont}{\on{c}}
\nc{\Supp}{\on{Supp}}
\nc{\blt}{\bullet}  
\nc{\dom}{\on{dom}}
\nc{\scon}{{\on{sc}}} 
\nc{\Affine}{\on{Aff}} 
\nc{\nscrA}{\mathscr{A}^{\on{nr}}} 
\nc{\nfraka}{{\bbf^{\on{nr}}}}
\nc{\ran}{{\rangle}}
\nc{\lan}{{\langle}}
\nc{\bk}{{\bar{k}}}
\nc{\tF}{{\tilde{F}}}
\nc{\sS}{{\bar{S}}}
\nc{\LG}{{^\text{L}\hspace{-0.04cm}G}}
\nc{\LL}{{^\text{L}\hspace{-0.07cm}L}}

\nc{\pot}[1]{ [\hspace{-0,5mm}[ {#1} ]\hspace{-0,5mm}] }
\nc{\rpot}[1]{ (\hspace{-0,7mm}( {#1} )\hspace{-0,7mm}) }

\nc{\defined}{\hspace{0.1cm}\stackrel{\text{\tiny \rm def}}{=}\hspace{0.1cm}}

\begin{document}

\title[Fibers of the Hodge-Newton Map for \(\GL_n\)]%
{On the geometry of certain non-basic affine Deligne-Lusztig varieties}
\author[Zhiming Li]{Zhiming Li}

\maketitle

\begin{abstract} 

Let \(F\) be a non-Archimedean local field, let \(L=\breve F\), and let
\(G=\GL_n\). Let \(M\subset G\) be a standard Levi subgroup and let
\(b\in M(L)\) be basic in \(M\), but not necessarily basic in \(G\). For a
dominant cocharacter \(\mu\), we study the reduction-to-Levi morphism
\[
  \beta:X^G_\mu(b)\longrightarrow
  \bigsqcup_{\mu_M\in S_M(\mu,\nu_b)}X^M_{\mu_M}(b)
\]
for affine Deligne--Lusztig varieties in the affine Grassmannian.

Using an Iwasawa factorization relative to \(P=MN\), we reduce the fiber
condition to explicit Frobenius-twisted lattice equations in the off-block
coordinates. In the Drinfeld case, where the base \(X^M_{\mu_M}(b)\) is
zero-dimensional, we prove that \(\beta\) is globally trivial with constant
affine-space fiber in the non-basic cases under consideration. More generally,
in the minuscule case we develop a nonzero-slope lattice-theoretic criterion
which shows that the fibers are affine spaces and that \(\beta\) is Zariski
locally a trivial affine-space bundle under natural slope-compatibility
hypotheses. We also give examples in the non-minuscule setting where the fibers
need not be affine spaces.

\end{abstract}

\tableofcontents

\thispagestyle{empty}

\section{Introduction}

Let \(F\) be a non-Archimedean local field, with ring of integers
\(\mathcal O_F\), fixed uniformizer \(t\), and residue field
\(\kappa=\mathcal O_F/(t)\). Let
\[
  L=\breve F
\]
be the completion of the maximal unramified extension of \(F\), with ring of
integers \(\mathcal O_L\), and let
\[
  \sigma=\sigma_{L/F}
\]
denote the Frobenius automorphism.

Throughout this paper we work with
\[
  G=\GL_n
\]
over \(F\). We fix the standard split maximal torus \(T\subset G\), the standard
Borel subgroup \(B\subset G\) of upper triangular matrices, and write
\(K=G(\mathcal O_L)\). For \(b\in G(L)\) and a dominant cocharacter
\(\mu\in X_\bullet(T)\), the affine Deligne--Lusztig variety in the affine
Grassmannian is
\[
  X^G_\mu(b)(\bar\kappa)
  =
  \left\{
    x\in G(L)/K:
    x^{-1}b\sigma(x)\in Kt^\mu K
  \right\}.
\]
We view \(X^G_\mu(b)\) as a perfect \(\bar\kappa\)-scheme locally of finite type.

Affine Deligne--Lusztig varieties were introduced as affine analogues of
classical Deligne--Lusztig varieties \cite{DL76}. Their non-emptiness and
dimension are closely related to the Newton point of \(b\), the Kottwitz
classification of \(\sigma\)-conjugacy classes \cite{Kot97}, and the dimension
formula studied by Görtz--Haines--Kottwitz--Reuman and Viehmann
\cite{GHKR06,Vie06}. In the basic case, the geometry of these varieties is
often closely related to the geometry of Rapoport--Zink spaces and Shimura
varieties; for example, Fox and Imai study irreducible components in the
\(GU(2,n-2)\) case \cite{FI22}. The aim of this paper is instead to study a
class of non-basic affine Deligne--Lusztig varieties for \(G=\GL_n\), by
reducing their geometry to the basic affine Deligne--Lusztig varieties inside
suitable Levi subgroups.

Let \(M\subset G\) be a standard Levi subgroup, and suppose that
\(b\in M(L)\) is basic in \(M\). There is a natural reduction-to-Levi morphism
\[
  \beta:
  X^G_\mu(b)
  \longrightarrow
  \bigsqcup_{\mu_M\in S_M(\mu,\nu_b)}
  X^M_{\mu_M}(b),
\]
as in \cite[Proposition~5.6.1]{GHKR06}. 

Especially, for our $G = \GL_n$ case, if \(\mu_M\) is the \(M\)-dominant minuscule cocharacter determined by \(b\), and set
$\mu=(\mu_M)_{\mathrm{dom}}$, then one can show that $S_M(\mu,\nu_b) = \{\mu_M\}$, which is the ``minuscule" situation that we mainly deal with throughout the paper. 

The philosophy is that the affine Deligne--Lusztig varieties on the right are
better understood because \(b\) is basic in \(M\). Thus the geometry of
\(X^G_\mu(b)\), where \(b\) may be non-basic in \(G\), can be approached by
studying the fibers of \(\beta\).

For \(G=\GL_n\), we exploit the Iwasawa decomposition relative to a parabolic
\(P=MN\). In the two-block case
\[
  M=\GL_{n_1}\times \GL_{n_2},
\]
a point in a fiber of \(\beta\) can be written using an off-block coordinate
\(X\in L^{n_1\times n_2}\). If the base point is represented by
\((A,B)\in \GL_{n_1}(L)\times \GL_{n_2}(L)\), then the fiber condition becomes
a lattice equation of the form
\[
  M_1\sigma(X)-XM_2\in \Lambda,
\]
where
\[
  M_1=A^{-1}b_1\sigma(A),
  \qquad
  M_2=B^{-1}b_2\sigma(B),
\]
and where \(\Lambda\) is an explicit lattice determined by the Cartan condition
\(g^{-1}b\sigma(g)\in Kt^\mu K\). Thus the study of \(\beta\) is reduced to
Frobenius-twisted linear algebra over \(L\).

We now state the main results of the paper. The first concerns the Drinfeld
case, by which we mean the case where the basic affine Deligne--Lusztig variety
on the Levi side is zero-dimensional.

\begin{thm}[Drinfeld case]\label{thm:intro-drinfeld}
Let
\[
  G=\GL_{n_1+\cdots+n_r},
  \qquad
  M=\GL_{n_1}\times\cdots\times\GL_{n_r}
\]
be a standard Levi subgroup, and let \(b\in M(L)\) be basic in \(M\), and that M is maximal with this
property. Let \(\mu_M\) be the \(M\)-dominant minuscule cocharacter determined by \(b\),
and set
\[
\mu=(\mu_M)_{\mathrm{dom}},
\] and that
\[
  X^M_{\mu_M}(b)
\]
is zero-dimensional. In the non-basic Drinfeld cases treated in Sections
\(3\)--\(5\), the reduction-to-Levi morphism
\[
  \beta:X^G_\mu(b)\longrightarrow X^M_{\mu_M}(b)
\]
is globally trivial, and its geometric fibers are affine spaces of explicitly
computed dimension.
\end{thm}

More concretely, in the two-block case \(M=\GL_n\times \GL_m\), the fiber
dimension is computed by tracing the cyclic Frobenius congruences among the
entries of the off-block matrix. For the canonical Drinfeld configuration
\(b=\operatorname{diag}(\tau_n,\tau_m)\), one obtains
\[
  \dim \beta^{-1}(x)=\min\{n,m\}-1
\]
for every geometric point \(x\in X^M_{\mu_M}(b)\), where the fibers are affine spaces except that in the equal-size
auxiliary case \(n=m\) additional discrete factors appear. The higher-block
case is obtained by iterating the two-block analysis.

The second main result treats the non-Drinfeld minuscule case. The key input is
a general nonzero-slope criterion for Frobenius-twisted lattice quotients.

\begin{thm}[Nonzero-slope affine quotient]\label{thm:intro-affine-quotient}
Let \((V,\Phi)\) be a simple \(F\)-isocrystal of nonzero slope, and set
\[
  f=\Phi-\mathrm{id}.
\]
Let \(\Lambda_0\subset V\) be a lattice, let \(\Lambda\subset V\) be an
\(\mathcal O_L\)-lattice, and let \(e\in\mathbb Z\) be such that
\[
  f(t^e\Lambda_0)\subset \Lambda.
\]
Then the quotient
\[
  f^{-1}(\Lambda)/t^e\Lambda_0
\]
is an affine space over \(\bar\kappa\), as a reduced variety. If the slope of
\(\Phi\) is positive, then \(f\) induces an isomorphism of
\(\bar\kappa\)-schemes
\[
  f^{-1}(\Lambda)/t^e\Lambda_0
  \xrightarrow{\sim}
  \Lambda/f(t^e\Lambda_0).
\]
\end{thm}

Applying this theorem to the Hom-isocrystals which occur in the off-block
coordinates gives the general minuscule statement.

\begin{thm}[General minuscule case]\label{thm:intro-minuscule}
Let
\[
  G=\GL_{k_1+\cdots+k_N},
  \qquad
  M=\GL_{k_1}\times\cdots\times\GL_{k_N}
\]
be a standard Levi subgroup. Let \(b=\operatorname{diag}(b_1,\dots,b_N)\in M(L)\)
be basic in \(M\), and assume that \(M\) is maximal with this property. Let \(\mu_M\) be the \(M\)-dominant minuscule cocharacter determined by \(b\), and set
\[
  \mu=(\mu_M)_{\mathrm{dom}}.
\]
Then the
reduction-to-Levi morphism
\[
  \beta:X^G_\mu(b)\longrightarrow X^M_{\mu_M}(b)
\]
is Zariski locally a trivial affine-space bundle.
\end{thm}

We do not claim in this generality that \(\beta\) is globally trivial or that it
is necessarily a vector bundle. An affine-space bundle becomes a vector bundle
only after showing that the transition functions are linear rather than merely
affine. Such a structure is available in certain canonical or Drinfeld cases,
but it is not automatic in the general non-Drinfeld setting.

Finally, in Section \(6\), we discuss examples in the non-minuscule case. These
examples show that the minuscule hypothesis is essential for the affine-space
behavior: outside the minuscule setting, factors such as \(\mathbb G_m\) may
appear, and the fibers need not be affine spaces.

\medskip

We now recall the basic notation used throughout the paper. Write \(B(G)\) for
the set of \(\sigma\)-conjugacy classes in \(G(L)\), and let
\[
\nu=\nu(\,\cdot\,):B(G)\longrightarrow (X_\bullet(T)_{\mathbb{Q}})^{\Gamma}_{\mathrm{dom}}
\]
be the Newton map to dominant, $\Gamma$-invariant rational cocharacters. Here
\[
\Gamma:=\Gal(L/F)=\Gal(\breve F/F)\cong\widehat{\mathbb Z},
\]
topologically generated by the arithmetic Frobenius $\sigma=\sigma_{L/F}$.  The
group $\Gamma$ acts on $X_\bullet(T)$ via the natural Galois action on $T_L$,
and $(\ \cdot\ )^{\Gamma}$ denotes the $\Gamma$-invariants. Since \(G=\GL_n\) is split, there is no nontrivial Galois
action on \(X_\bullet(T)\). For \(b\in G(L)\), we write
\[
  \nu_b:=\nu([b])
\]
for its Newton point. The \(\sigma\)-centralizer of \(b\) is the algebraic group
over \(F\) defined by
\[
  J_b(R)
  =
  \left\{
    g\in G(L\otimes_F R):
    g^{-1}b\sigma(g)=b
  \right\}.
\]
Its defect is
\[
  \operatorname{def}_G(b)
  =
  \operatorname{rank}_F G-\operatorname{rank}_F J_b.
\]

We say that \(b\) is \emph{basic} if its Newton point is central, i.e.
\[
  \nu_b\in X_\bullet(Z(G))_{\mathbb Q}.
\]
For \(G=\GL_n\), this is equivalent to saying that the associated isocrystal is
isoclinic. We say that \(b\) is \emph{superbasic} if it is not
\(\sigma\)-conjugate into any proper Levi subgroup of \(G\). For \(\GL_n\), this
means that the associated isocrystal is simple.

\paragraph{\textbf{A simplest version of classical Deligne-Lusztig varieties.}}
Let $G$ be an unramified connected reductive group over the finite field $\mathbb{F}_q$, $B\subset G$ a Borel subgroup with flag variety $G/B$, and $\Fr$ the $q$-Frobenius. For a Weyl group element $w\in W$, the \emph{classical Deligne-Lusztig variety} is
\[
X_w(1)\;:=\;\bigl\{\,gB\in G(\overline{\mathbb{F}}_q)/B:\ g^{-1}\Fr(g)\in BwB\,\bigr\}.
\]
These varieties are smooth, quasi-projective, and have
\[
\dim X_w(1)=\ell(w),
\]
the length of $w$. Deligne and Lusztig used $X_w(1)$ for $w$ in the finite Weyl group to classify the irreducible representations of the finite group of Lie type $G(\mathbb{F}_q)$. Affine Deligne-Lusztig varieties (ADLVs) are the natural $p$-adic/affine analogues obtained by replacing $B$ with a parahoric (e.g.\ $K=G(\mathcal O_L)$) and allowing a general element $b$; they arose from studying the geometry of Shimura varieties over finite fields.

\begin{fact}{Rapoport's dimension formula (see \cite[Conjecture~1.0.1]{GHKR06} and \cite{Vie06}).}
Let $G$ be an unramified connected reductive group,  and $\mu\in X_\bullet(T)$ dominant.  If $X^G_{\mu}(b)\neq\varnothing$, then
\begin{equation}\label{eq:Rapoport-dim}
\dim X^G_{\mu}(b)\;=\;\big\langle \rho,\;\mu-\nu_b\big\rangle\;-\;\frac{1}{2}\,\mathrm{def}_G(b),
\end{equation}
where $\rho$ is the half-sum of the positive roots (with respect to $B$). 
\end{fact}
In particular, in situations where the right-hand side of \eqref{eq:Rapoport-dim} vanishes, the corresponding affine Deligne-Lusztig variety is zero-dimensional. For a Levi $M \subset G$, and $b \in M(L)$ basic in $M$, and a $M$-dominant minuscule cocharacter $\mu$ with $X^M_{\mu}(b)\neq\varnothing$, if $\dim X^M_{\mu}(b) = 0$, we refer to this zero-dimensional basic locus in a Levi \(M\) as the \emph{Drinfeld case} for \((M,\mu,b)\).
\medskip
\ \\
Let \(G = \GL_n\) and set a standard Levi subgroup $M$ of $G$ by:
\[
  M = \prod_{i=1}^r \GL_{n_i},\qquad \sum_{i=1}^r n_i = n,
\]
so that $M$ is ``$r$-block'' standard. Note that in this case $\Gamma$ acts on $X_\bullet(T)$ trivially. Let $b \in M(L)$ be an $M$-basic element, and $\mu\in X_\bullet(T)$ be a $G$-dominant cocharacter with $X^G_\mu(b)\neq\varnothing$. Let \(P=MN\) be the standard parabolic with Levi factor \(M\). Following Görtz--Haines--Kottwitz--Reuman,
we define \(S_M(\mu)\) to be the set of \(M\)-dominant cocharacters
\(\mu_M\) such that
\[
  (\mu_M)_{\text{dom}}\preceq\mu,\qquad N(L)x_{\mu_M}\cap Kx_\mu\neq\varnothing
\]
inside the affine Grassmannian \(G(L)/K\). Here
\[
  x_\lambda=t^\lambda K
\]
for a cocharacter \(\lambda\).

Let
\[
  p_M:X_\bullet(T)\longrightarrow \Lambda_M
\]
be the natural projection to the quotient of \(X_\bullet(T)\) by the coroot
lattice of \(M\). For \(\nu\in\Lambda_M\), set
\[
  S_M(\mu,\nu)
  :=
  \{\mu_M\in S_M(\mu):p_M(\mu_M)=\nu\}.
\]
In the reduction-to-Levi morphism, the relevant value of \(\nu\) is the
Kottwitz invariant of \(b\) in \(M\), equivalently the image of \(b\) under
\[
  \eta_M:M(L)\longrightarrow \Lambda_M.
\]
Thus the reduction-to-Levi map has target
\[
  \beta:
  X^G_\mu(b)
  \longrightarrow
  \bigsqcup_{\mu_M\in S_M(\mu,\eta_M(b))}
  X^M_{\mu_M}(b).
\]
%\[ %\Sigma(\mu)_M = \{
%\mu_M \in X_\bullet(T): \mu_M \text{\ is $M$-dominant}, (\mu_M)_{\text{dom}}\preceq\mu\},\] and the set \[ S_M(\mu) = \{
%\mu_M \in\Sigma(\mu)_M: \mu_M \text{\ is a maximal under $\preceq_M$ in $ \Sigma(\mu)_M$}\}.\]
%We then define
%\[S_M(\mu, v_b) = \{
%\mu_M \in S_M(\mu): X^M_{\mu_M}(b)\neq\varnothing\}.\]
%Notice that the projection of \(b\) onto each \(\GL_{n_i}\), denoted \(b_i\), is basic in \(\GL_{n_i}\). If there is an \(M\)-minuscule cocharacter \(\mu_{0}\in S_M(\mu,v_b)\) such that \((\mu_0)_{\mathrm{dom}}=\mu\), then \(\mu_0\) is the unique element of \(S_M(\mu,v_b)\). The reduction morphism
Notice first that the projection of \(b\) to each factor \(\GL_{n_i}\), say
\(b_i\), is basic in \(\GL_{n_i}\).  Let \(\nu_{b_i}=\lambda_i\cdot(1,\dots,1)\) be
its Newton point, so \(n_i\lambda_i=v(\det b_i)\).  For any
\(\mu_M\in S_M(\mu,v_b)\) write \(\mu_M=(\mu^{(1)},\dots,\mu^{(r)})\) with
\(\mu^{(i)}\in X_\bullet(\GL_{n_i})\) dominant.  The non-emptiness condition for
\(X^M_{\mu_M}(b)\) forces the blockwise sum constraint
\[
\langle \mathbf{1}_{n_i},\,\mu^{(i)}\rangle
=\sum_{a=1}^{n_i}\mu^{(i)}_a
= n_i\lambda_i
= v(\det b_i)\qquad(1\le i\le r).
\]
If further $\mu^{(i)}$ is minuscule in $\GL_{n_i}$, one must have $\mu^{(i)} = (\underbrace{1+q_i,\dots,1+q_i}_{p_i},q_i,\dots,q_i\bigr)$, where $q_i = \lfloor \lambda_i\rfloor$, $p_i = n_i(\lambda_i - \lfloor \lambda_i\rfloor)$. In what follows we are primarily interested in the case where \(\mu_M\) is
\(M\)-minuscule and \((\mu_M)_{\mathrm{dom}}=\mu\). In this case, for any element $\mu' \in S_M(\mu,v_b)$ and any $1 \leq i \leq r$, its $\GL_{n_i}$ component has the same weight as the $\GL_{n_i}$ component of $\mu_M$. The $\GL_{n_i}$ component of $\mu_M$ is minuscule for all $i$ so one has $\mu_M \preceq_M \mu'$, and hence by \((\mu')_{\text{dom}}\preceq\mu = (\mu_M)_{\text{dom}}\) one has $\mu_M = \mu'$. Hence in this case the set \(S_M(\mu,v_b)\) is a
singleton: there is a unique \(\mu_M\) compatible with \((\mu,b)\). The reduction morphism
\[
  \beta:\ X^G_{\mu}(b)\longrightarrow X^M_{\mu_M}(b)
\]
(on the unique non-empty summand) is induced by the canonical projection $G \twoheadrightarrow M$ coming from the Iwasawa decomposition relative to $P=MN$.

Our first goal is, in the Drinfeld case for $(M,\mu,b)$:
\begin{enumerate}
  \item Reformulate $\dim \beta^{-1}(x)$ for every geometric point $x\in X^M_{\mu_M}(b)$, which should be consistent with the Rapoport dimension formula.
  \item Give conditions under which $\beta$ is (Zariski) locally trivial, and in favorable cases a \emph{vector bundle}, with explicitly computed rank.
\end{enumerate}

Since any $F$-Levi subgroup of $G$ is $G(F)$-conjugate to a standard one, the statements hold for a general Levi subgroup of $G$. 

\medskip
\section{Background for the ADLV on $\GL_n$}
\paragraph{\textbf{Basic and superbasic elements in $\GL_n$.}}
For $G=\GL_n$, an element $b\in \GL_n(L)$ is \emph{basic} iff the isocrystal $(L^n,b\sigma)$ is \emph{isoclinic}, i.e. all its Newton slopes are equal: $\nu_b=\lambda\cdot(1,\ldots,1)$ with a single $\lambda\in \mathbb{Q}$.  It is \emph{superbasic} iff it is not $\sigma$-conjugate into any proper Levi of $\GL_n$, equivalently, the isocrystal is simple (no nontrivial slope-stable decomposition).

Let $\tau_n\in\GL_n(L)$ be the cyclic shift matrix defined by
\[
\tau_n e_1=t\,e_n,\qquad \tau_n e_i=e_{i-1}\ (2\le i\le n).
\]
Then one checks
\[
(\tau_n\,\sigma)^n = t\cdot\sigma^n,\qquad (t^q\tau_n^p\,\sigma)^n=t^{qn+p}\cdot\sigma^n,
\]
so $(L^n,\,t^q\tau_n^p\sigma)$ has single slope $\lambda=q+\tfrac{p}{n}$.  By the 
%Dieudonn\'e-Manin 
classification (see \cite{Kot97}), every isoclinic $n$-dimensional isocrystal of slope $\lambda$ is isomorphic to one of the form $(L^n,\,t^q\tau_n^p\sigma)$ with $0\le p<n$ and $q=\lfloor \lambda\rfloor$; in particular:
\begin{itemize}
  \item If $b$ is \emph{superbasic}, then $b$ is $\sigma$-conjugate to $t^q\tau_n^p$ with $\gcd(p,n)=1$ (the simple cyclic block of dimension $n$).
  \item In general, if $b$ is \emph{basic} of slope $\lambda$, then $b$ is $\sigma$-conjugate into a standard Levi to a block diagonal $\bigoplus_i t^{q_i}\tau_{s_i}^{p_i}$ with all ratios $p_i/s_i=\lambda-\lfloor\lambda\rfloor$ equal; thus $b$ is built from the same cyclic shift/wrap pattern as powers of $\tau_{\bullet}$.
\end{itemize}
Consequently, for purposes of our combinatorial/block computations on affine Deligne-Lusztig varieties, basic elements in $\GL_n$ may be treated as (block) powers of $\tau_n$, with the central factor $t^q$ merely shifting the (central) Newton slope and not affecting the cyclic index bookkeeping.\\
\begin{comment}
\noindent\textbf{Background for \(\GL_n\).}
For \(\GL_n\), \(\sigma\)-conjugacy classes are classified by isocrystals up to isomorphism, hence by Newton slopes.  An element \(b\in\GL_n(L)\) is basic iff its isocrystal is isoclinic (all slopes equal).  In this case, there are integers \(q\in\mathbb Z\) and \(p\) with \(0\le p<n\) such that
\[
b\ \sim_{\sigma}\ t^{\,q}\,\tau_n^{\,p},
\]
where \(\tau_n\) is the cyclic (companion) matrix with \(t\) in the \((n,1)\)-entry and \(1\) on the superdiagonal.  Thus, up to central powers of \(t\), a basic class is represented by a power of \(\tau_n\). 
\end{comment}
In the Drinfeld situation for \(\GL_n\) (i.e.\ the minuscule case with zero-dimensional basic locus), the Euclidean decomposition \(p\equiv 1\) or \(n-1\pmod n\) is forced; in other words, modulo inversion of the cyclic direction, one is essentially in the two cases \(b=\tau_n\) or \(b=\tau_n^{\,n-1}\). Taking \(\mu\) to be the corresponding minuscule coweight (so that the Hodge point matches the Newton slope in the basic setting), the condition
\(
g^{-1}b\,\sigma(g)\in K\,t^\mu K = K\,b\,K
\)
is equivalent to \(g^{-1}b\,\sigma(g)\in J_b\cdot K\), and the natural map
\[
J_b \longrightarrow X^{{\GL_n}}_{\mu}(b),\qquad j\longmapsto jK,
\]
induces an identification
\(
X^{{\GL_n}}_{\mu}(b)\cong J_b/(J_b\cap K),
\)
i.e.\ the basic locus in the Drinfeld case is a single \(J_b\)-orbit of hyperspecial cosets.

\subsection{Iwasawa decomposition}

Let \(P=MN\) be the standard parabolic with Levi \(M\) and unipotent radical \(N\).  By Iwasawa,
\[
  G(L) = P(L)\,K = M(L)\,N(L)\,K,
\]
and recall that the reduction to Levi map $\beta$ is defined by the natural projection: $$X^G = G(L)/K \longrightarrow M(L)/K_M = X^M.$$ From now on, for any algebraic subgroup $G'$ of $G$, we will simply denote the maximal compact subgroup $G'_K$ also by $K$ if there is no confusion.\\
\ \\
We first establish the argument for a ``2-block" standard Levi $M$ in $G$. Let \(G=\GL_{n+m}\), \(M=\GL_n\times\GL_{m}\), and $b \in M(L)$ a basic element in $M$. Write $b = \mathrm{diag}(b_1, b_2)$, and take $\mu_i$ to be the unique dominant minuscule cocharacter having the same totoal weight as $b_i$ in the corresponding block, and finally take $\mu = (\mu_1 \oplus \mu_2)_{\mathrm{dom}}$. Fix a representative $A\oplus B$ such that $\bigl(A\oplus B\bigr)\bmod K \in X^{\GL_n}_{\mu_1}(b_1)\;\times\;X^{\GL_{m}}_{\mu_2}(b_2)$, then any coset \(gK\in \beta^{-1}(A\oplus B)\) can be written as
\[
  g = \begin{pmatrix}
    A & C \\[4pt]
    0 & B
  \end{pmatrix},
  \quad
  %A\in\GL_k(L),\;
  %B\in\GL_{n-k}(L),\;
  C\in L^{\,n\times m}.
\]
 We try to use $C$ above to parametrize the fiber $\beta^{-1}(A\oplus B)$, then restricting to \(gK\in X^G_\mu(b)\) mapping to \(\bigl(A\oplus B\bigr)\bmod K\)
imposes two conditions on $C$:
\begin{enumerate}
  \item[(i)] When will two parameters $C_1$, $C_2$ represent the same coset in \(G(L)/K\),
  \item[(ii)] the admissibility \(g^{-1}b\,\sigma(g)\in K\,t^\mu\,K\).
\end{enumerate}
\subsection{Coset parametrization and admissibility}

Write
\[
  g = \begin{pmatrix}A & C\\[4pt]0 & B\end{pmatrix},
  \quad
  gK\in\beta^{-1}(A\oplus B).
\]
\begin{enumerate}
  \item \emph{Coset identification.}
  Two such blocks \(C_1,C_2\) define the same coset iff
  \[
    \begin{pmatrix}A & C_1\\0 & B\end{pmatrix}K
    =
    \begin{pmatrix}A & C_2\\0 & B\end{pmatrix}K
    \;\Longleftrightarrow\;
    A^{-1}(C_1-C_2)\in \calO_L^{\,n\times m}.
  \]

  \item \emph{Admissibility.}
  Computing
  \[
    g^{-1}
    = \begin{pmatrix}A^{-1}&-A^{-1}CB^{-1}\\[4pt]0&B^{-1}\end{pmatrix}, 
    \quad
    \sigma(g)
    = \begin{pmatrix}\sigma(A)&\sigma(C)\\[4pt]0&\sigma(B)\end{pmatrix},
  \]
  one finds
  \[
    g^{-1}b\,\sigma(g)
    =
    \begin{pmatrix}
      A^{-1}b_1\,\sigma(A) &
      A^{-1}\bigl(b_1\,\sigma(C) - C\,B^{-1}b_2\,\sigma(B)\bigr)\\[6pt]
      0 & B^{-1}b_2\,\sigma(B)
    \end{pmatrix}.
  \]

  Now we rewrite $g$ as
  \[
    g \;=\;
    \begin{pmatrix}A & C\\[4pt]0 & B\end{pmatrix}
    \;=\;
    (A\;\oplus\;B)\;\ltimes\;C',
    \qquad
    C' := A^{-1}C\;\in\;L^{n\times m}.
  \]
  Set
  \[
    M_1 \;=\; A^{-1}\,b_1\,\sigma(A)\;\in\; Kt^{\mu_1}K,
    \qquad
    M_2 \;=\; B^{-1}\,b_2\,\sigma(B)\;\in\;Kt^{\mu_2}K.
  \]
  Then the congruent condition for the new parameter $C'$ is rewritten as: \[
C'_1 \sim C'_2
\quad\Longleftrightarrow\quad
C'_1 - C'_2 \;\in\;\mathcal{O}_L^{\,n\times m}.\] and the admissibility condition
  \(\;g^{-1}b\,\sigma(g)\in K\,t^\mu\,K\)
  is equivalent to force 
$$\displaystyle
\begin{pmatrix}
M_1 & M_1\,\sigma(C') - C'\,M_2\\[3pt]
0   & M_2
\end{pmatrix} \in K\,t^\mu K.
$$
\end{enumerate}
  This reformulation in terms of \(C'\) and the conjugates \(M_1,M_2\) makes it clear that the fiber is cut out by a system of linear congruences in the entries of \(C'\). We firstly recall the following lemma: 
\noindent

\begin{lem}[Smith normal form criterion for Cartan decomposition]\label{lem:smith normal form criterion}
For any \(g\in\GL_n(L)\), define the ideal
\[
\Delta_k(g)
=(\text{all }k\times k\text{ minors of }g)\;\subset\;\calO_L,
\quad
k=1,\dots,n,
\]
Given integers 
\[
a_1\le a_2\le \cdots\le a_n,
\]
then
\[
g\;\in\;K\,\diag\bigl(t^{a_1},\dots,t^{a_n}\bigr)\,K
\]
if and only if
\[
v\bigl(\Delta_k(g)\bigr)
\;=\;
a_1 + a_2 + \cdots + a_k
\quad
\text{for all }k=1,\dots,n.
\]
\end{lem}
\ \\
%Write $h = M_1\,\sigma(C') - C'\,M_2 = (h_{ij})$, 
Then to study what conditions should be imposed on $C'$, we need to apply the criterion of Lemma 2.1 for the given $M_1$, $M_2$.
\begin{cor}\label{cor:block‐monomial}
Let 
\[
g=\begin{pmatrix}A& C\\[4pt]0& B\end{pmatrix}
\;\in\;\GL_{n+m}(L),
\]
where \(A\in\GL_n(L)\) and \(B\in\GL_m(L)\) are \emph{monomial} matrices, i.e. each row of \(A\) and each column of \(B\) contains exactly one nonzero entry.  For \(1\le i\le n\), let \(\alpha_i=v(A_{i,\sigma(i)})\) be the valuation of the unique nonzero in row \(i\) of \(A\), and for \(1\le j\le m\), let \(\beta_j=v(B_{\tau(j),j})\) be the valuation of the unique nonzero in column \(j\) of \(B\).  Order the combined list
\(\{\alpha_1,\dots,\alpha_n,\beta_1,\dots,\beta_m\}\)
into 
\(\gamma_1\le\gamma_2\le\cdots\le\gamma_{n+m}\).  Then
\[
g\;\in\;
K\,\diag\bigl(t^{\gamma_1},\dots,t^{\gamma_{n+m}}\bigr)\,K
\]
if and only if for every entry \(c_{ij}\) of the off-diagonal block \(C\), one has
\[
v\bigl(c_{ij}\bigr)\;\ge\;
\min\{\alpha_i,\;\beta_j\}.
\]
\end{cor}
\begin{proof}[Proof of Corollary~\ref{cor:block‐monomial}]
By left and right multiplication with permutation matrices in \(K\) (which preserves the Cartan double coset and the determinantal ideals \(\Delta_k\)), we may assume
\[
A=\diag(t^{\alpha_1},\dots,t^{\alpha_n}),\qquad
B=\diag(t^{\beta_1},\dots,t^{\beta_m}),
\]
so that
\[
g=
\begin{pmatrix}
\diag(t^{\alpha_1},\dots,t^{\alpha_n}) & C\\[3pt]
0 & \diag(t^{\beta_1},\dots,t^{\beta_m})
\end{pmatrix}.
\]
Let \(\Gamma=\{\alpha_1,\dots,\alpha_n,\beta_1,\dots,\beta_m\}\) and write its nondecreasing rearrangement as \(\gamma_1\le\cdots\le\gamma_{n+m}\).

\smallskip
\noindent{\em Sufficiency.}
Assume that for all \(i,j\) one has
\[
v(c_{ij})\;\ge\;\min\{\alpha_i,\beta_j\}.
\]
Fix \(k\) and consider any nonzero \(k\times k\) minor of \(g\). Such a minor necessarily selects some set \(I\subset\{1,\dots,n\}\) of rows from the top block and some set \(J\subset\{1,\dots,m\}\) of rows from the bottom block, with \(|I|+|J|=k\), together with the same number of columns; because the lower-left block is zero, any chosen left (resp. right) column must be matched with a top (resp. bottom) row, except possibly when a top row is paired with a right column via an entry of \(C\).
Expanding the determinant as a sum over matchings, every monomial term is a product of factors, each of which is either \(t^{\alpha_i}\) (if a top row \(i\) is matched to its left column \(i\)), or \(t^{\beta_j}\) (if a bottom row \(j\) is matched to its right column \(j\)), or \(c_{ij}\) (if a top row \(i\) is matched to a right column \(j\) through \(C\)). By the hypothesis \(v(c_{ij})\ge\min\{\alpha_i,\beta_j\}\), the valuation of any such product is bounded below by the sum of the corresponding \(\alpha\)- and \(\beta\)-contributions in which each \(c_{ij}\) is replaced by \(\min\{\alpha_i,\beta_j\}\). In particular, the minimal possible valuation among all terms is achieved by the “block-diagonal” matching that avoids \(C\) altogether and pairs \(s\) top rows with \(s\) left columns and \(k-s\) bottom rows with \(k-s\) right columns, for some \(0\le s\le k\). Hence
\[
v\bigl(\Delta_k(g)\bigr)
\;=\;
\min_{0\le s\le k}
\Bigl(\alpha_{(1)}+\cdots+\alpha_{(s)}\;+\;\beta_{(1)}+\cdots+\beta_{(k-s)}\Bigr),
\]
where \(\alpha_{(r)}\) (resp.\ \(\beta_{(r)}\)) denotes the \(r\)-th smallest \(\alpha\) (resp.\ \(\beta\)).
This minimum is exactly \(\gamma_1+\cdots+\gamma_k\). By the Smith normal form criterion (Lemma~\ref{lem:smith normal form criterion}), we conclude
\[
g\in K\,\diag\bigl(t^{\gamma_1},\dots,t^{\gamma_{n+m}}\bigr)\,K.
\]

\smallskip
\noindent{\em Necessity.}
Conversely, suppose there exist indices \((i_0,j_0)\) with
\[
v(c_{i_0 j_0})\;<\;\min\{\alpha_{i_0},\beta_{j_0}\}.
\]
We will produce a \(k\times k\) minor whose valuation is strictly smaller than \(\gamma_1+\cdots+\gamma_k\), contradicting the Smith criterion.

Let \(S\) be the multiset of the \(k-1\) smallest elements of \(\Gamma\setminus\{\max(\alpha_{i_0},\beta_{j_0})\}\); that is, from \(\Gamma\) remove the larger of \(\alpha_{i_0},\beta_{j_0}\) and take the \(k-1\) smallest remaining elements. Then
\[
\gamma_1+\cdots+\gamma_k \;=\; \min\{\alpha_{i_0},\beta_{j_0}\} \;+\; \sum_{x\in S} x.
\]
Form the \(k\times k\) submatrix by choosing:
\begin{itemize}
  \item the top row \(i_0\) and the right column \(n+j_0\) (thus contributing the entry \(c_{i_0j_0}\));
  \item for each \(\alpha_i\in S\), the top row \(i\) and the left column \(i\) (contributing \(t^{\alpha_i}\));
  \item for each \(\beta_j\in S\), the bottom row \(n+j\) and the right column \(n+j\) (contributing \(t^{\beta_j}\)).
\end{itemize}
Because all other chosen entries lie on the block diagonal and the lower-left block is zero, this submatrix is upper triangular up to a single off-diagonal \(c_{i_0j_0}\), so its determinant is (up to a unit)
\[
\det = c_{i_0j_0}\cdot \prod_{\alpha_i\in S} t^{\alpha_i}\cdot \prod_{\beta_j\in S} t^{\beta_j},
\]
and therefore
\[
v(\det)\;=\;v(c_{i_0j_0})\;+\;\sum_{x\in S}x
\;<\;\min\{\alpha_{i_0},\beta_{j_0}\}\;+\;\sum_{x\in S}x
\;=\;\gamma_1+\cdots+\gamma_k.
\]
Thus \(v(\Delta_k(g))<\gamma_1+\cdots+\gamma_k\), contradicting the Cartan/Smith criterion. Hence necessarily \(v(c_{ij})\ge\min\{\alpha_i,\beta_j\}\) for all \(i,j\).

\smallskip
Combining the two directions completes the proof.
\end{proof}
\medskip
\ \\
We give an example below to help readers better understand the above Corollary.\\
\begin{ex}{($A=\tau_2$, $B=\tau_3$).}
Recall
\[
\tau_2=\begin{pmatrix}0&1\\ t&0\end{pmatrix},\qquad
\tau_3=\begin{pmatrix}0&1&0\\ 0&0&1\\ t&0&0\end{pmatrix}.
\]
These are monomial. By rows of $A$ we get
\[
\alpha_1=v(1)=0,\quad \alpha_2=v(t)=1,
\]
and by columns of $B$ we get
\[
\beta_1=v(t)=1,\quad \beta_2=v(1)=0,\quad \beta_3=v(1)=0.
\]
The multiset $\{\alpha_1,\alpha_2,\beta_1,\beta_2,\beta_3\}=\{0,1,1,0,0\}$ rearranges to
\[
\gamma_\bullet=(0,0,0,1,1).
\]
Hence the corollary says that
\[
g=\begin{pmatrix}A&C\\ 0&B\end{pmatrix}\in
K\,\diag(1,1,1,t,t)\,K
\quad\Longleftrightarrow\quad
v(c_{ij})\ge \min\{\alpha_i,\beta_j\}.
\]
Writing the $2\times3$ matrix $C=(c_{ij})$, the inequalities are
\[
\begin{array}{c|ccc}
 & j=1 & j=2 & j=3\\ \hline
i=1 & v(c_{11})\ge 0 & v(c_{12})\ge 0 & v(c_{13})\ge 0\\
i=2 & v(c_{21})\ge 1 & v(c_{22})\ge 0 & v(c_{23})\ge 0
\end{array}
\]
Thus \emph{only} the lower-left entry must be divisible by $t$; all other
off-block entries need only be integral. If, for instance, $v(c_{21})=0$, one can
choose a $4\times4$ minor using $c_{21}$ whose valuation drops below
$\gamma_1+\cdots+\gamma_4=1$, contradicting the Smith normal form criterion.
\end{ex}

\section{Triviality of \(\beta\) for the Drinfeld Case in a $2$-Block Levi $M$}

Let \(r> 1\) and recall that \(\tau_r\in\GL_r(L)\) denotes the \emph{canonical superbasic element}. By definition, one computes
\[
\tau_r \;=\;
\begin{pmatrix}
0      & 1      &        &   &   \\[-3pt]
\vdots       & \ddots & \ddots &   &   \\
0       &        & \ddots     & 1 &   \\[-3pt]
t       &0        & \cdots   & 0 &   
\end{pmatrix},
\]
i.e.\ the unique \(r\times r\) matrix with \(t\) in the lower-left corner, \(1\) on the super-diagonal, and \(0\) elsewhere. Moreover notice that $\tau_1 = (t) \in \GL_1(L)$.\\
\subsection{The canonical Drinfeld case}
We firstly study the case where $b$ is canonically superbasic in a $2$-Block Levi $M$, say the component of $b$ in each $r \times r$ block is $\tau_r$. We retain all notation from above: \(G=\GL_{n+m}\), \(M=\GL_n\times\GL_{m}\), and for the canonical superbasic
\[
b = \diag(\tau_n,\tau_{m})\;\in\;M(L)
\]
we write \(\mu=(1,1,0,\dots,0)\) and \(\mu_M=(1,0,\dots,0)\oplus(1,0,\dots,0)\).  Recall that for the superbasic element \(\tau_r\in\GL_r(L)\) and the minuscule coweight \(\omega_r=(1,0,\dots,0)\), the affine Deligne-Lusztig variety
\[
X^{\GL_r}_{\omega_r}(\tau_r)
\;=\;
\bigl\{\,gK\in\GL_r(L)/K\mid g^{-1}\tau_r\,\sigma(g)\in K\,t^{\omega_r}\,K\bigr\}
\]
identifies canonically with the \(\sigma\)-centralizer quotient
\[
X^{\GL_r}_{\omega_r}(\tau_r)
\;\cong\;
J_{\tau_r}\big/ \bigl(J_{\tau_r}\cap K\bigr).
\]
In particular, in our case
we may fix
\[
M_1=\tau_n\in\GL_n(L), 
\qquad
M_2=\tau_{m}\in\GL_{m}(L).
\]
Notice that we are mainly interested in the case where $b$ is NOT basic in $G$, in which case $m \neq n$, but now we also include the case $m = n$ for the study of our whole setting which can be seen later.
%once and for all to represent the two basic loci
%\(\;X^{\GL_n}_{\omega_n}(\tau_n)\) and \(X^{\GL_{m}}_{\omega_{m}}(\tau_{m})\).  
\\
Writing
\[
C' \;=\;A^{-1}C\;\in\;L^{\,n\times m},
\quad
h \;=\;M_1\,\sigma(C') \;-\; C'\,M_2
=\bigl(h_{ij}\bigr)_{1\leq i\leq n, 1\leq j\leq m},
\]
Applying Corollary~\ref{cor:block‐monomial},  \(\;g^{-1}b\,\sigma(g)\in K\,t^\mu\,K\) iff:
\[
h_{n,1}
\;=\;
(\tau_n\,\sigma(C')-C'\,\tau_{m})_{n,1}
\;\in\;t\,\calO_L,
\quad
h_{ij}\;\in\;\calO_L
\;\text{for all other }(i,j).
\]
\ \\
\begin{thm}\label{thm:two-block triviality}
Fix integers $n, m \geq 1$. Let \(G=\GL_{n+m}\), \(M=\GL_n\times\GL_{m}\), and set
\[
b=\diag\bigl(\tau_n,\tau_{m}\bigr),
\quad
\mu=(1,1,0,\dots,0).
\]
Then for every point
\((A\oplus B)\bmod K\in X^M_{\mu_M}(b)\), the fiber of
\(\beta\colon X^G_\mu(b)\to X^M_{\mu_M}(b)\) has dimension
\[
\dim_{\bar \kappa}\,\beta^{-1}(A\oplus B)
\;=\;
\min\{n,m\}-1,
\]
in exact agreement with the Rapoport dimension formula
\[
\dim(X^G_\mu(b)) = \bigl\langle\rho,\mu-\bar v_b\bigr\rangle
\;-\;\tfrac12\,\mathrm{def}_G(b).
\]
Moreover, as varieties over \(\bar \kappa\), $\beta$ is a trivial bundle with fiber
\[
\begin{cases}
\displaystyle
\mathbb A^{\,\min\{n,m\}-1}
&\text{if \(n \neq m\),}\\[1ex]
\displaystyle
\mathbb A^{\,k-1}\;\times\;
\Bigl(L^{\sigma^k}/\calO_{L^{\sigma^k}}\Bigr)^k
&\text{if \(n = m = k\).}
\end{cases}
\]
\end{thm}
\medskip
\ \\
Before proving the main theorem, we introduce a convenient shorthand and establish an integrality lemma.

\begin{dfn}
  For elements \(w,v\in L\), we write
  \[
    w \overset{\sigma} \longrightarrow  v
    \quad\Longleftrightarrow\quad
    \sigma(w) \;-\;v\;\in\;\calO_L .
  \]
  Since for any $w \in L$, $\sigma$ has finite order on the image $\bar{w} \in L/\calO_L$, one sees that the relation $\overset{\sigma} \longrightarrow$ generates an equivalent relation $\overset{\sigma}{\thicksim}$ on $L/\calO_L$, and each equivalence class of it has finitely many elements.\\
  \ \\
   For elements \(w,v\in L\), we write
  \[\begin{aligned}
    w \overset{t\sigma} \longrightarrow  v
    \quad\Longleftrightarrow\quad
    t\sigma(w) \;-\;v\;\in\;\calO_L , \\    
    w \overset{t^{-1}\sigma} \longrightarrow  v
    \quad\Longleftrightarrow\quad
    \sigma(w) \;-\;tv\;\in\;\calO_L .
  \end{aligned} 
  \] It is easy to see that $\overset{t\sigma} {\longrightarrow}$, $\overset{t^{-1}\sigma} {\longrightarrow}$ induce a reverse relation pair on $(L/\calO_L)/\overset{\sigma} {\thicksim}$\ \ which we denote by $\overset{t} {\longrightarrow}$, $\overset{t^{-1}} {\longrightarrow}$ respectively. Moreover, one has \[(w\overset{t} {\longrightarrow} v) \wedge (w\overset{t} {\longrightarrow}  u) \quad\Longrightarrow\quad v = u,\] as well as \[w\overset{t} {\longrightarrow} w\quad\Longleftrightarrow\quad w = \calO_L, \text{\ the trivial class},\] that is, the relation $\overset{t} {\longrightarrow}$ extends to a cotree order (i.e. each element in the set has at most $1$ successor) on $(L/\calO_L)/\overset{\sigma} {\thicksim}$ with the unique terminal element $\calO_L$.
\end{dfn}

\begin{lem}[Integrality on a Cyclic Tree]\label{lem:cyclic-tree}
Let \(R=\{r_1,\dots,r_N\}\subset L\) be a finite set whose images in \(L/\calO_L\) form a cycle under the relations \(\overset{\sigma}{\longrightarrow}\), \(\overset{t\sigma}{\longrightarrow}\), or \(\overset{t^{-1}\sigma}{\longrightarrow}\).  Suppose furthermore that \(R\) admits a \emph{piecewise partition}
\[
R \;=\;(\bigsqcup_{i=1}^r W_i)\sqcup (\bigsqcup_{i=1}^s S_i)
\]
into disjoint subchains \(W_i\) and \(S_i\) satisfying:

\begin{enumerate}
  \item[(1)] Within each \(W_i\) or \(S_i\), every arrow is of type \(\overset{\sigma}{\longrightarrow}\).
  \item[(2)] For each adjacent pair \(W_i, W_j\), they are connected by an arrow of type \(\overset{\sigma}{\longrightarrow}\) or \(\overset{t^{-1}\sigma}{\longrightarrow}\) ($\overset{t\sigma}{\longrightarrow}$ resp.).
  \item[(3)] In each \(S_i=\{s_{i,1},\dots,s_{i,r_i}\}\), the first element \(s_{i,1}\) admits an incoming arrow of type \(\overset{t^{-1}\sigma}{\longrightarrow}\), and the last element \(s_{i,r_i}\) admits an outgoing arrow of type \(\overset{t\sigma}{\longrightarrow}\).
  \item[(4)] At least one adjacent pair\((W_i,W_j)\) admits an arrow of type \(\overset{t^{-1}\sigma}{\longrightarrow}\) ($\overset{t\sigma}{\longrightarrow}$ resp.).
\end{enumerate}

\noindent Then:

\begin{itemize}
  \item Every element of each \(W_i\) belongs to \(\calO_L\).
  \item Within each \(S_i\), once one entry is fixed modulo \(\calO_L\), all others are uniquely determined modulo \(\calO_L\).  Equivalently, the set of solutions in \(S_i\) is a torsor under the one-dimensional \(\bar{\kappa}\)-vector space \(t^{-1}\calO_L/\calO_L\).
\end{itemize}
%In other words, all $W_i$ are in the same $\sigma$-equivalent class $\calO_L$, and each $S_i$ lies in a %successor of this class.
\end{lem}

\begin{proof}
By (3), each subchain \(S_i\) is "sandwiched" between a unique incoming arrow of type \(\overset{t^{-1}\sigma}{\longrightarrow}\) and a unique outgoing arrow of type \(\overset{t\sigma}{\longrightarrow}\). These two arrows necessarily connect \(S_i\) to (resp.\ from) the same \(\sigma\)-equivalence class among the \(W_j\) by the uniqueness of the immediate successor.

Now remove all the \(S_i\) from the original cyclic diagram in \(R\), replacing each by a finite concatenation of pure \(\overset{\sigma}{\longrightarrow}\) arrows. The result is a smaller loop whose arrows are only of type \(\overset{\sigma}{\longrightarrow}\) or \(\overset{t^{-1}\sigma}{\longrightarrow}\) ($\overset{t\sigma}{\longrightarrow}$ resp.). Tracing this loop from any element necessarily returns to the starting point, and hence by (4), each such parameter must coincide with its own \(t^{-1}\sigma\)-successor ($t\sigma$-successor resp.), forcing it into \(\mathcal{O}_L\).

Finally, re-inserting each \(S_i\) between its prescribed incoming and outgoing arrows shows that each \(S_i\) contributes exactly one free class in \(t^{-1}\mathcal{O}_L/\mathcal{O}_L\). These yield the independent \(\bar\kappa\)-parameters claimed, completing the proof.
\end{proof}

\subsection*{Proof of Theorem 3.1 in the Case \(n \neq m\)}
\begin{proof}
Write \(C'=(c_{ij})\) for \(1\le i\le n\), \(1\le j\le m\), and extend the indices periodically by  
\[
c_{i,j}=c_{i',j'}
\quad\text{where }i'\equiv i\pmod{n},\;
j'\equiv j\pmod{m}.
\]
Then the admissibility conditions
\[
h_{ij}=(\tau_n\,\sigma(C')-C'\,\tau_{m})_{ij}\;\in\;\calO_L
\quad(\forall(i,j)\neq(n,1)),
\qquad
h_{n1}\;\in\;t\,\calO_L
\]
translate into the relations
\[\begin{aligned}
\sigma(c_{i,j})-c_{i-1,j-1}&\in\calO_L
\quad(1<i\le n,\;1<j\le m),\\
t\,\sigma(c_{1,j})-c_{n,j-1}&\in\calO_L
\quad(1<j\le m),\\
\sigma(c_{i,1})-t\,c_{i-1,m}&\in\calO_L
\quad(1<i\le n),\\
t\,\sigma(c_{1,1})-t\,c_{n,m}&\in t\,\calO_L
\;\Longleftrightarrow\;
\sigma(c_{1,1})-c_{n,m}\in\calO_L.
\end{aligned}\]
Starting from any $c_{ij}$ and for each time we subtract both indexes by $1$ , we then get a cyclic loop:
\[
c_{i,j}\xrightarrow{} c_{i-1,j-1}\xrightarrow{}\cdots
c_{i+1,j+1}\xrightarrow{}c_{i,j},
\]
where the arrows are $\xrightarrow{\sigma}$, $\xrightarrow{t\sigma}$ or $\xrightarrow{t^{-1}\sigma}$. By certain symmetric argument of the following one, without loss of generality, we can assume $n < m$. For each \(1\le i\le n-1\), define the $i$-th diagonal block parameter set
\[
S_i \;=\;\bigl\{\,c_{j,\;m-i+j}\;\bigm|\;1\le j\le i\bigr\},
\]
i.e., the set of the $i$-th ``strict upper off-diagonal'', and for \(n\le i\le m \), define the $i$-th diagonal block parameter set 
\[
W_i \;=\;\bigl\{\,c_{j,\;m-i+j}\;\bigm|\;1\le j\le n\bigr\},
\]
i.e., the set of the $i$-th ``diagonal other than the strict upper off-diagonals''. Then the loop above consists of some specific $W_i$'s and $S_i$'s. A direct inspection shows that for any given loop above, the hypotheses (1)-(3) of Lemma 3.3 hold for the collection \(\{W_i\}\) and \(\{S_i\}\) that appear in the loop with the arrow in hypotheses (2) the $\xrightarrow{t\sigma}$ (the assumpion $n < m$ plays a role here: if instead $n > m$, the arrow in hypotheses (2) is $\xrightarrow{t^{-1}\sigma}$). We have to check the hypothesis (4): There is at least one adjacent pair \((W_i,W_j)\) that admits an arrow of type \(\overset{t\sigma}{\longrightarrow}\) in the loop. Such an arrow can only from $c_{1,j}$ to $c_{n,j-1}$ for some $1<j\leq m$, and $c_{1,j}$, $c_{n,j-1}$ can both be contained in $W_i$'s iff $j \leq m-n+1$. Therefore it suffices to show that the loop contains $c_{1,j}$ for some $1<j\leq m-n+1$. Write $d = \gcd(n,m)$, then it is easy to observe that two parameters $c_{1,j}, c_{1,j'}$ belong to the same loop iff $j\equiv j'\pmod{d}$. Since $m > n$, one has that the list $2, \cdots, m-n+1$ is a full list of representatives in $\mathbb{Z}/d\mathbb{Z}$, so it has nontrivial intersection with any loop above.
\begin{comment}Notice that the length of the loop above is $k(n-k)/\gcd(k, n-k) > k$, while $|S_i| = i < k$ and $|W_i|\leq k$, so the loop contains at least two diagonal sets. Assume the loop does not contain any $S_i$, so we have at least $2$ arrows joining some $W_i$'s, but the only arrow that is ``\(\overset{\sigma}{\longrightarrow}\)" between $W_i$'s is from $c_{1,1}$ to $c_{k,n-k}$, so (4) is satisfied.
Assume the loop contains some $S_i$, choose $S_d$ with $d$ the maximum among those $S_i$
\end{comment}
\\
Therefore by Lemma 3.3, each \(w\in W_i\) lies in \(\calO_L\), and within each \(S_i\) there is a single free class in \(t^{-1}\calO_L/\calO_L\).  Hence the total number of free parameters is  
\(\sum_{i=1}^{n-1}1=n-1\),  
as claimed.  Moreover, we can choose the coordinates $c_{1, m}, \cdots,c_{n-1,m}$ uniformly as the free parameters so that we see $\beta$ is trivial with fiber $\mathbb A^{n-1}$.
\end{proof}
\ \\
Next we establish a discreteness lemma to prove the remaining part of Theorem 3.1.
\begin{lem}[Discreteness on a Cyclic Tree]\label{lem:cyclic-tree}
Let
\[
R=\{r_1,\dots,r_N\}\;\subset\;L
\]
be a finite set whose images in \(L/\calO_L\) form a single cycle under the relations
\(\overset{\sigma}{\longrightarrow}\), \(\overset{t\sigma}{\longrightarrow}\), \(\overset{t^{-1}\sigma}{\longrightarrow}\).  Suppose furthermore that \(R\) admits a decomposition into disjoint subchains
\[
R \;=\;(\bigsqcup_{i=1}^r W_i)\sqcup (\bigsqcup_{i=1}^s S_i)
\]
satisfying the conditions (1)-(3) of Lemma~3.3, and in addition

\begin{enumerate}
  \item[(4')]  Every adjacent pair \(\,(W_i, W_{i+1})\) in the cycle is connected by an arrow of type \(\overset{\sigma}{\longrightarrow}\).
\end{enumerate}

Let \(d = |R|\) be the length of the cycle.  Then for any chosen base-point
\(\,w_0\in W_i\), the iterated congruences around the loop give
\[
\sigma^d(w_0)\;-\;w_0\;\in\;\calO_L.
\]
 Moreover, once \(w_0\) is fixed modulo \(\calO_L\), all other elements of every \(W_i\) are uniquely determined, while each $S_i$ carries a free parameter identified with $(t^{-1}w_0+t^{-1}\calO_L)/\calO_L$.  In particular, the solution space for \(w_0\) is $L^{\sigma^d}/\calO_{L^{\sigma^d}}$, and the total solution set of $R$ is $\bar{\kappa}^{|\{S_i\}|}\times L^{\sigma^d}/\calO_{L^{\sigma^d}}$, which is a finite dimensional affine $\bar{\kappa}$-space times a discrete set.
\end{lem}

\begin{proof}
First, replace each \(s_i\in S_i\) by \(\tilde s_i = t\,s_i\).  Under this rescaling every \(\overset{t\sigma}{\longrightarrow}\) or \(\overset{t^{-1}\sigma}{\longrightarrow}\) arrow becomes a plain \(\overset{\sigma}{\longrightarrow}\) arrow, so the entire cycle on
\(\;\widetilde R=\bigl(\bigsqcup W_i\bigr)\sqcup\bigl(\bigsqcup\tilde S_i\bigr)\)
is now a \(\sigma\)-cycle of length \(d\).  Fix a base-point \(w_0\in W_i\); following the unique \(\sigma\)-arrows around the loop yields  
\[
\sigma^d(w_0)\;-\;w_0\;\in\;\calO_L.
\]
Each \(W_i\)-element is then uniquely recovered from \(w_0\) via successive \(\sigma\)-arrows, and each original \(s_i\) is determined by  
\((t^{-1}\tilde s_i+t^{-1}\calO_L)/\calO_L\),  
with \(\tilde s_i\) itself uniquely determined by $w_0 \pmod {\calO_L}$. Again, inside one $S_i$, once a base-point $s_0$ is fixed, the other $s_i$'s are determined by the $\sigma$-arrows. Consequently the full solution set is a product of one copy of the discrete group \(L^{\sigma^d}/\calO_{L^{\sigma^d}}\) (parametrizing \(w_0\)) and one copy of \(\bar{\kappa}\) for each \(S_i\), as claimed.
\end{proof}

\subsection*{Proof of Theorem 3.1 in the Case \(n = m = k\)}
\begin{proof}
Retain the notation and periodic indexing from the \(n\neq m\) case.  Observe that the index range  
\[
1<j\le m-n+1
\;=\;1
\]
is empty, so no \(c_{1,j}\) ever appears and hence the condition (4') in Lemma 3.4 holds. Notice that each loop has length 
\[
\frac{n\,m}{\gcd(n,m)}=\frac{k^2}{k}=k.
\]
There are exactly \(k\) loops: one “pure” diagonal loop \(W_k = \{c_{1,1},c_{2,2},\dots,c_{k,k}\}\) and \(k-1\) “off-diagonal’’ loops each of which traverses a single \(W_i\) and a single $S_j$.  

By Lemma~3.4 each off-diagonal loop yields a one-dimensional affine \(\bar\kappa\)-parameter together with the discrete factor \(L^{\sigma^k}/\calO_{L^{\sigma^k}}\).  The pure diagonal loop, having no \(S_i\), contributes only the discrete quotient \(L^{\sigma^k}/\calO_{L^{\sigma^k}}\).  Altogether there are \(k-1\) nontrivial loops, together with a discrete factor $(L^{\sigma^k}/\calO_{L^{\sigma^k}})^k$, which can be parametrized uniformly by $c_{1,k}, \cdots c_{k,k}$, as required.
\end{proof}

\medskip

We compute some easy examples below to help readers understand the above argument better.

\begin{ex}\label{ex:GL2}
  %(\(\GL_2\)-case.)
  Take \(n=1\), \(m=1\), so 
  \(G=\GL_2\), \(M=\GL_1\times\GL_1\).  Fix
  \[
    b_1=t,\quad b_2=1,\quad \mu=(1,0).
  \]
  Then \(A,B\in L^\times\) satisfy \(A^{-1}t\sigma(A)\in K\,tK\), \(B^{-1}\sigma(B)\in K\). Recall that $X^{\GL_1}_{(1)}(1) = F/\calO_F \cong \mathbb{Z}$, thus we may take $A = t^{r_1}, B = t^{r_2}$. Then the parameter \(C\in L\) is subject to
  \[
    %b_1\,\sigma(C)-C\,b_2
    %\;=\;
    t\,\sigma(C)-C
    \;\in\;
    A\,\calO_L.
  \]
  Write \(v(-)\) for the valuation on \(L\).  Since
  \(\;v\bigl(t\,\sigma(C)-C\bigr)=\min\{v(C)+1,\;v(C)\}=v(C),\)
  the condition
  \(t\,\sigma(C)-C\in A\calO_L\)
  is equivalent to
  \[
    v(C)\;\ge\;v(A).
  \]
  Hence \(C\in A\calO_L\), and modulo the equivalence \(C\sim C'\) when \(C-C'\in A\calO_L\), the fiber is
  \[
    A\calO_L / A\calO_L
    \;\cong\;
    \{0\},
  \]
  a single point.  In particular
  \(\dim\beta^{-1}(A\oplus B)=0\).
\end{ex}
\bigskip

\begin{ex} Take \(G=\GL_3\), \(n=2, m=1\).
  Let 
  \[
    b_1=\begin{pmatrix}0&1\\t&0\end{pmatrix}\in\GL_2(L),
    \quad
    b_2=t\in L^\times,
    \quad
    \mu=(1,1,0).
  \]
  Any element of the fiber above \((A\oplus B)\bmod K\in X^M_{\mu_M}(b)\) may be written as
  \[
    g \;=\;
    \begin{pmatrix}A & C\\[4pt]0 & B\end{pmatrix}
    \;=\;
    (A\;\oplus\;B)\;\ltimes\;C',
    \qquad
    C' := A^{-1}C\;\in\;L^{2\times1}.
  \]
  Set
  \[
    M_1 \;=\; A^{-1}\,b_1\,\sigma(A)\;\in\; K\begin{pmatrix}t & 0\\[4pt]0 & 1\end{pmatrix}K,
    \qquad
    M_2 \;=\; B^{-1}\,b_2\,\sigma(B)\;\in\;KtK.
  \]
Notice that 
\[
  X^{\GL_2}_{(1,0)}(b_1)
  \;=\;
  X_{b_1}(b_1)
  \;=\;
  J_{b_1}/\bigl(J_{b_1}\cap K\bigr)
\]
is itself a basic locus, so we may take \(M_1=b_1\) and \(M_2=t\). \\
  \ \\
  Then the congruence condition is rewritten as: \[
C'_1 \sim C'_2
\quad\Longleftrightarrow\quad
C'_1 - C'_2 \;\in\;\mathcal{O}_L^{\,2\times1}.\] and the admissibility condition
  \(\;g^{-1}b\,\sigma(g)\in K\,t^\mu\,K\)
  is equivalent to the single matrix-equation
  \[
    M_1\,\sigma(C') \;-\; C'\,M_2 = h = (h_{ij})
    \;\in\;(\calO_L)_{11} \oplus (t\calO_L)_{21}.
  \]

 Writing 
\(\;C'=(c_1,c_2)^T\), the admissibility condition
\[
  M_1\,\sigma(C') \;-\; C'\,M_2
  \;=\;
  b_1\,\sigma(C') \;-\; t\,C'
  \;=\;
  \begin{pmatrix}
    \sigma(c_2)-t\,c_1 \\[4pt]
    t\bigl(\sigma(c_1)-c_2\bigr)
  \end{pmatrix}
  \;\in\;(\calO_L)_{11} \oplus (t\calO_L)_{21}
\]
implies
\(c_1\in \calO_L\) and \(c_2\in\calO_L\). Therefore the dimension of $\beta^{-1}((A\oplus B)\bmod K)$ is $0$. \\
  
\end{ex}
\bigskip
\begin{ex}[recover a known basic example] Take \(G=\GL_2\), \(n=m=1\).
  Let 
  \[
    b_1=b_2 =1,
    \quad
    \mu=(0, 0) = 1.
  \]
  Here one sees $b$ is in fact basic in $G$ but we can still use the off-block analysis. With the same notations as before, it is obvious that \(M_1=1\) and \(M_2=1\), and the admissibility condition
  \(\;g^{-1}b\,\sigma(g)\in K\,t^\mu\,K = K\)
  is
  \[
    \sigma(C') \;-\; C'\
    \;\in\; \calO_L.
  \]

We know that $\{c \in L: \sigma(c)-c \in \calO_L\} = \calO_L+F$, therefore we recover the fact: $$X^{\GL_2}_{1}(\text{Id}) = \GL_2(F)/\GL_2(\calO_F),$$
hence $\beta$ is trivial with fiber isomorphic to $F/\calO_F$, which is a discrete set and has dimension $0$.
  
\end{ex}
\bigskip
\begin{ex}[a non-trivial example]
    Set \(G=\GL_5\), \(n=2, m=3\), with 
\[
b_1=\begin{pmatrix}0&1\\t&0\end{pmatrix},\quad
b_2=\begin{pmatrix}0&1&0\\0&0&1\\t&0&0\end{pmatrix},\quad
\mu=(1,1,0,0,0),
\]

Again we have
\[
  X^{\GL_2}_{(1,0)}(b_1)
  \;=\;
  X_{b_1}(b_1)
  \;=\;
  J_{b_1}/\bigl(J_{b_1}\cap K\bigr) 
   \;\ \ \ \ \;
  X^{\GL_3}_{(1,0,0)}(b_2)
  \;=\;
  X_{b_2}(b_2)
  \;=\;
  J_{b_2}/\bigl(J_{b_2}\cap K\bigr)
\]
are themselves basic loci.\\

Thus we can again take \(M_1=b_1\), \(M_2=b_2\) and write 
\[
g=\begin{pmatrix}A&C\\0&B\end{pmatrix}
=A\oplus B\ltimes C',
\qquad
C'=A^{-1}C\;\in\;L^{2\times3}.
\]
The admissibility condition \(g^{-1}b\,\sigma(g)\in K\,t^\mu K\) becomes
\[
b_1\,\sigma(C')\;-\,C'\,b_2\;\in\;\calO_L^{5} \oplus (t\calO_L)_{21}.
\]
Write $C' = (c_{ij})$, one computes
\[
b_1\,\sigma(C') \;-\; C'\,b_2
=\begin{pmatrix}
\sigma(c_{21})-t\,c_{13} & \sigma(c_{22})-c_{11} & \sigma(c_{23})-c_{12}\\[4pt]
t\,\sigma(c_{11})-t\,c_{23} & t\,\sigma(c_{12})-c_{21} & t\,\sigma(c_{13})-c_{22}
\end{pmatrix}.
\]
The ``loop tracing" argument implies that $c_{13} \in t^{-1}\calO_L$, and the other $c_{ij}$'s lie in $\calO_L$. Therefore the solution space is one-dimensional over $\bar{\kappa}$ with free parameter $c_{13}$. In particular the
map \(\beta\) is a trivial vector bundle over the basic loci $X^{\GL_2}_{(1,0)}(b_1)\times X^{\GL_3}_{(1,0,0)}(b_2)$.

\end{ex}
\ \\
\subsection{The general Drinfeld case}
Next we treat the general Drinfeld case in \(M=\GL_n\times\GL_m\).  Recall that for a superbasic $b$ in $M(L)$, we can write
\[
b=\bigl(\tau_n^{\,d_1},\,\tau_m^{\,d_2}\bigr),
\]
with \(\gcd(d_1,n)=\gcd(d_2,m)=1\).  Since \(\tau_r^r=t\Id_r\), perform the Euclidean divisions
\[
d_1=q_1n+p_1,\quad 0<p_1<n,
\qquad
d_2=q_2m+p_2,\quad 0<p_2<m,
\]
so that
\[
b=\bigl(t^{q_1}\,\tau_n^{\,p_1},\;t^{q_2}\,\tau_m^{\,p_2}\bigr),
\]
still with \(\gcd(p_1,n)=\gcd(p_2,m)=1\). \\
\\
The unique \(M\)-dominant minuscule \(\mu_M\) is then
\[
\mu_M
=\bigl(\underbrace{1+q_1,\dots,1+q_1}_{p_1},q_1,\dots,q_1\bigr)
\;\oplus\;
\bigl(\underbrace{1+q_2,\dots,1+q_2}_{p_2},q_2,\dots,q_2\bigr),
\]
and we take \(\mu\) to be its \(G\)-dominant rearrangement.  In the Drinfeld case, we have $\dim X^M_{\mu_M}(b) = 0$, which forces $p_1 \in  \{1, n-1\}$ and $p_2 \in \{1, m-1\}$. 
\begin{comment}One checks, as in the canonical superbasic case, that admissibility
\[
g^{-1}b\,\sigma(g)\;\in\;K\,t^\mu\,K
\]
forces
\[
h_{n,1}\in t\,\mathcal O_L,
\qquad
h_{ij}\in\mathcal O_L\quad(i,j)\neq(n,1),
\]
where \(h=\tau_n^{\,d_1}\,\sigma(C')-C'\,\tau_m^{\,d_2}\).  The same “two-block” loop-tracing argument then shows \(\beta\) is a trivial affine bundle of the predicted rank.  
\end{comment}
\begin{thm}\label{thm:drinfeld-q1neqq2-zero-fiber}(layered case)
In the setting of Drinfeld case above, assume \(q_1\neq q_2\). Then for every
\((A,B)\bmod K\in X^M_{\mu_M}(b)\) the fiber
\[
\beta^{-1}(A,B)\;\subset\;X^G_{\mu}(b)
\]
is zero-dimensional (indeed, a single \(K\)-coset over \(\bar \kappa\)).
\end{thm}

\begin{proof}
In the Drinfeld case setting of $\GL_n$, we can write
\[
M_1=t^{q_1}\tau_n^{p_1},\qquad M_2=t^{q_2}\tau_m^{p_2},
\]
with $p_1 \in  \{1, n-1\}$ and $p_2 \in \{1, m-1\}$.  Without loss of generality assume \(q_1<q_2\) and set \(\delta:=q_2-q_1\ge 1\).  As in Section 2, put \(C'=A^{-1}C\) and
\[
h \;=\; M_1\,\sigma(C')-C'\,M_2.
\]
By Corollary~\ref{cor:block‐monomial} together with the Drinfeld (minuscule) shape of \(\mu_M\), the admissibility condition \(g^{-1}b\,\sigma(g)\in K\,t^\mu K\) forces all entries of \(h\) to have valuation \(\ge q_1\), and the entries with minimal possible valuation are exactly those lying in \(t^{q_1}\calO_L\).  Factoring out \(t^{q_1}\) we obtain the normalized congruence system
\begin{equation}\label{eq:normalized}
\widetilde h\;:=\;\tau_n^{p_1}\,\sigma(C')\;-\;t^{\delta}\,C'\,\tau_m^{p_2}
,
\end{equation}
with some entries lying in \(t\calO_L\) and all others in \(\calO_L\).

Now recall the shape of \(\tau_r^{p}\): it is the permutation matrix for the shift by \(p\) with a factor \(t\) exactly at the wrap positions; more precisely, multiplication on the left by \(\tau_n^{p_1}\) sends an entry of \(C'\) either to a \emph{unit} multiple (coefficient \(1\)) of a \(\sigma\)-shifted entry or to a \emph{uniformizer} multiple (coefficient \(t\)) of a \(\sigma\)-shifted entry, depending on whether the row-shift wraps. Likewise, multiplication on the right by \(\tau_m^{p_2}\) sends an entry either to a unit multiple or to a \(t\)-multiple, depending on whether the column-shift wraps.

Consequently, each scalar congruence extracted from \eqref{eq:normalized} has the form
\begin{equation}\label{eq:edge}
u\,\sigma(c_{i,j})\;-\;t^\delta\,v\,c_{i',j'}\;\in\; t^{\mathrm{val}(u)}\calO_L,
\qquad u,v\in\{1,t\},
\end{equation}
linking a pair of entries \((i,j)\mapsto(i',j')\) determined by the shifts \(p_1\) and \(p_2\).  Rewriting \eqref{eq:edge} gives
\[
\sigma(c_{i,j}) \;\equiv\; t^{k}\,c_{i',j'} \pmod{\calO_L},
\qquad
k=\delta+v_t-u_t\in\{\delta-1,\delta,\delta+1\},
\]
where \(u_t=1\) if \(u=t\) and \(0\) otherwise (similarly for \(v_t\)).  Since \(\delta\ge1\), we always have \(k\ge0\); moreover \(k=0\) can occur only when \(\delta=1\) and simultaneously \(u=t, v=1\).  In all other cases \(k\ge1\).

View the index set \(\{(i,j)\}\) as vertices of a directed graph whose edges are the relations \((i,j)\to(i',j')\) furnished by \eqref{eq:edge}.  Because the loop advances by \((-p_1,-p_2)\) modulo \((n,m)\), and $p_1 = 1, n-1$ and $p_2 = 1, m-1$, we have $i' = i \pm1$, $j' = j\pm 1$. Along any loop, compose the congruences to obtain
\begin{equation}\label{eq:loop}
\sigma^{r}(c)\;\equiv\;t^{K}\,c\pmod{\calO_L},
\end{equation}
for some length \(r\ge1\) and integer \(K=\sum k_\ell\) equal to the sum of the exponents \(k\) along the loop.  By the discussion above, every edge satisfies \(k_\ell\ge0\), and \(k_\ell=0\) can only happen when \(\delta=1\) and $\sigma(c_{i,j})$ is shifted by a uniformizer multiple while $c_{i',j'}$ is shifted by a unit multiple at that step.  It is impossible for \emph{every} step in a loop to satisfy this simultaneous wrap condition: indeed, the left wrap occurs precisely on a subset of \(p_1\) consecutive rows, and the right wrap on a subset of \(p_2\) consecutive columns; as the loop advances by either plus or minus one index modulo \((n,m)\), it must exit at least one of these wrap subsets.  Hence every loop contains at least one edge with \(k_\ell>0\), so \(K\ge1\).

From \eqref{eq:loop} we deduce \(t^{K}\,c-\sigma^{r}(c)\in\calO_L\) with \(K\ge1\), which forces \(c\in\calO_L\) (apply valuations and use that \(\sigma\) preserves \(v\)).  Since this holds for the representative \(c\) of each \(\sigma\)-orbit of entries in every loop, all entries of \(C'\) lie in \(\calO_L\).  By the coset identification \(C'_1\sim C'_2\iff C'_1-C'_2\in\calO_L^{n\times m}\), it follows that all admissible \(C'\) represent the same coset as \(0\), and hence the fiber \(\beta^{-1}(A,B)\) consists of a single \(K\)-coset, and hence \(\dim_{\bar\kappa}\beta^{-1}(A,B)=0\).
\end{proof}
\medskip
Therefore, the only potentially nontrivial situation for a \(2\)-block Levi \(M\) arises when \(q_1=q_2\). In this case the pair \((b,\mu)\) differs from \((\tau_n^{p_1}\oplus\tau_m^{p_2},\,\mu')\) by the central twist \(t^{q_1}\), so by dividing out this scalar we reduce to the case \(q_1=q_2=0\) without changing the fiber structure. The canonical superbasic case \(p_1=p_2=1\) has already been treated; thus it remains to analyze the complementary Drinfeld configurations
\[
(p_1,p_2)\in\{(n-1,m-1),\ (1,m-1),\ (n-1,1)\}.
\]
By symmetry between the two blocks, it suffices to handle the mixed cases \((1,m-1)\) and \((n-1,m-1)\); the remaining mixed case follows by swapping the roles of \(n\) and \(m\). We now turn to these three residual cases.

\subsubsection{The case \(p_1=n-1,\;p_2=m-1\)}
Write
\[
\tau_n^{\,n-1}
=t\cdot\tau_n^{-1}
=\begin{pmatrix}
0&0&\cdots&0&1\\
t&0&\cdots&0&0\\
0&t&\ddots&\vdots&\vdots\\
\vdots&\ddots&\ddots&0&0\\
0&\cdots&0&t&0
\end{pmatrix},
\qquad
\tau_m^{\,m-1}
=t\cdot\tau_m^{-1}
=\begin{pmatrix}
0&0&\cdots&0&1\\
t&0&\cdots&0&0\\
0&t&\ddots&\vdots&\vdots\\
\vdots&\ddots&\ddots&0&0\\
0&\cdots&0&t&0
\end{pmatrix}.
\]
Thus \(\tau_r^{\,r-1}\) is the “reverse cyclic shift’’: it shifts indices by \(-1\) with a factor \(t\) on the subdiagonal and a unit at the \((1,r)\)-entry.  The admissibility equation in this case reads
\[
h\;=\;\tau_n^{\,n-1}\,\sigma(C')\;-\;C'\,\tau_m^{\,m-1},
\]
and by Corollary~\ref{cor:block‐monomial} (applied with the monomial factors \(\tau_n^{\,n-1}\) and \(\tau_m^{\,m-1}\)) entrywise one has
\[
\begin{aligned}
h_{i,j}&=t\,\sigma(c_{i-1,j})-t\,c_{i,j+1} \in t\calO_L
&& (1<i\le n,\;1\le j<m),\\
h_{1,j}&=\sigma(c_{n,j})-t\,c_{1,j+1} \in \calO_L
&& (1\le j<m),\\
h_{i,m}&=t\,\sigma(c_{i-1,m})-c_{i,1} \in \calO_L
&& (1<i\le n),\\
h_{1,m}&=\sigma(c_{n,m})-c_{1,1}\,\in \calO_L.
\end{aligned}
\]
The integrality pattern is obtained from the canonical case \(p_1=p_2=1\) 
with the direction of the arrows reversed as well as the roles of the arrows of type \(t\sigma\) / \(t^{-1}\sigma\) interchanged.  Hence all loop-tracing and congruence arguments from the case \(p_1=p_2=1\) apply essentially the same way to \(C'\) here as before.  In particular, the fiber description, the number of free parameters, and the (local) triviality of \(\beta\) in this \((n-1,m-1)\) configuration agree with the canonical case, up to the above interchange.
\begin{thm}\label{thm:two-block triviality dual case}
Fix integers $n, m \geq 1$. Let \(G=\GL_{n+m}\), \(M=\GL_n\times\GL_{m}\), and set
\[
b=\diag\bigl(\tau^{n-1}_n,\tau^{m-1}_{m}\bigr),
\quad
\mu=(1,1,\dots,1,0,0).
\]
Then for every point
\((A\oplus B)\bmod K\in X^M_{\mu_M}(b)\), the fiber of
\(\beta\colon X^G_\mu(b)\to X^M_{\mu_M}(b)\) has dimension
\[
\dim_{\bar \kappa}\,\beta^{-1}(A\oplus B)
\;=\;
\min\{n,m\}-1,
\]
in exact agreement with the Rapoport dimension formula
\[
\dim(X^G_\mu(b)) = \bigl\langle\rho,\mu-\bar v_b\bigr\rangle
\;-\;\tfrac12\,\mathrm{def}_G(b).
\]
Moreover, as varieties over \(\bar \kappa\), $\beta$ is a trivial bundle with fiber
\[
\begin{cases}
\displaystyle
\mathbb A^{\,\min\{n,m\}-1}
&\text{if \(n \neq m\),}\\[1ex]
\displaystyle
\mathbb A^{\,k-1}\;\times\;
\Bigl(L^{\sigma^k}/\calO_{L^{\sigma^k}}\Bigr)^k
&\text{if \(n = m = k\).}
\end{cases}
\]
\end{thm}
\begin{proof}[Sketch of proof]
This is the ``dual'' version of Theorem~\ref{thm:two-block triviality}.  Write
\[
h\;=\;\tau_n^{\,n-1}\,\sigma(C')\;-\;C'\,\tau_m^{\,m-1}.
\]
As noted in the preceding subsubsection, the integrality pattern for \(h\) is obtained from the canonical case \(p_1=p_2=1\) 
with the direction of the arrows reversed as well as the roles of the arrows of type \(t\sigma\) / \(t^{-1}\sigma\) interchanged.  Consequently, all loop-tracing and
congruence arguments from Theorem~\ref{thm:two-block triviality} apply verbatim to \(C'\) as well.

Define the diagonal ``chains'' \(S_i\) and \(W_i\) exactly as in the canonical case:
\(S_i\) consists of the strictly upper off-diagonals of height \(i\)
(\(1\le i\le \min\{n,m\}-1\)), while \(W_i\) consists of the remaining diagonals
(including the possible main diagonal).  The coset-congruence \(C'_1\sim C'_2\iff
C'_1-C'_2\in\calO_L^{n\times m}\) and the admissibility bounds for \(h\) imply that
each \(W_i\) contributes no free \(\bar\kappa\)-parameter, whereas each \(S_i\) contributes exactly one.  Thus the number of free parameters is
\(\sum_{i=1}^{\min\{n,m\}-1}1=\min\{n,m\}-1\), realized explicitly, e.g., by the entries on the last column of the strictly upper off-diagonals
\[
c_{1,m},\ c_{2,m},\ \dots,\ c_{\,\min\{n,m\}-1,\;m}.
\]
The fiber description then reduces exactly the same way as in Theorem~\ref{thm:two-block triviality}.
\end{proof}
\subsubsection{The case \(p_1=1,\;p_2=m-1\)}
In this mixed Drinfeld configuration we have
\[
h\;=\;\tau_n\,\sigma(C')\;-\;C'\,\tau_m^{\,m-1}.
\]
Writing \(C'=(c_{ij})_{1\le i\le n,\;1\le j\le m}\), and using that
\(\tau_n\) shifts rows up by \(+1\) with a factor \(t\) on wrap
(\((n,1)\)-entry), while \(\tau_m^{\,m-1}\) shifts columns down by \(-1\) with
a factor \(t\) on the subdiagonal and a unit at \((1,m)\), one checks the
entrywise formulas
\[
\begin{aligned}
h_{i,j}&=\sigma\!\bigl(c_{i+1,j}\bigr)-t\,c_{i,j+1} \in \calO_L
&& (1\le i<n,\ 1\le j<m),\\
h_{n,j}&=t\,\sigma\!\bigl(c_{1,j}\bigr)-t\,c_{n,j+1} \in t\calO_L
&& (1\le j<m),\\
h_{i,m}&=\sigma\!\bigl(c_{i+1,m}\bigr)-c_{i,1} \in \calO_L
&& (1\le i<n),\\
h_{n,m}&=t\,\sigma\!\bigl(c_{1,m}\bigr)-c_{n,1} \in \calO_L.
\end{aligned}
\]
By inspecting the congruence relations above, the only \(t\sigma\)-edge in the
index diagram arises from
\(
h_{n,m}=t\,\sigma(c_{1,m})-c_{n,1}
\),
that is,
\[
c_{1,m}\ \xrightarrow{\,t\sigma\,}\ c_{n,1}.
\]
All other edges are of type \(\sigma\) or \(t^{-1}\sigma\). Consequently, except the degenerate cases that we list in the following theorem,
loop-tracing forces integrality to propagate through every cycle, and the only
free parameter occurs at the entry \(c_{1,m}\) (modulo \(\calO_L\)). Hence
\(c_{1,m}\) is the unique free variable, and the fiber
\(\beta^{-1}(A\oplus B)\) is an affine line over \(\bar\kappa\); in particular,
\(\dim_{\bar\kappa}\beta^{-1}(A\oplus B)=1\).
\begin{thm}\label{thm:two-block triviality mixed case}
Let \(n,m\ge 1\), \(G=\GL_{n+m}\), \(M=\GL_n\times\GL_m\), and consider the mixed Drinfeld configuration
\[
b=\diag\bigl(\tau_n,\ \tau_m^{\,m-1}\bigr),
\qquad
\mu=\bigl(\underbrace{1,\dots,1}_{m},0,\dots,0\bigr),
\]
with the analogous statement holding after swapping the two blocks
\(\bigl(\tau_n^{\,n-1},\,\tau_m\bigr)\).
Except for the degenerate case \(n=m=2\) (which reduces to the canonical case
\(p_1=p_2=1\)), the following holds: for every point
\((A\oplus B)\bmod K\in X^M_{\mu_M}(b)\), the fiber of
\(\beta\colon X^G_\mu(b)\to X^M_{\mu_M}(b)\) is an affine line over \(\bar\kappa\),
hence
\[
\dim_{\bar\kappa}\,\beta^{-1}(A\oplus B)=1,
\]
and \(\beta\) is (Zariski) a trivial \(\mathbb A^1\)-bundle over its image.
\end{thm}
\begin{proof}
Retain the mixed Drinfeld configuration and the entrywise relations
\[
\begin{aligned}
h_{i,j}&=\sigma(c_{i+1,j})-t\,c_{i,j+1}\in\calO_L
&& (1\le i<n,\ 1\le j<m),\\
h_{n,j}&=t\,\sigma(c_{1,j})-t\,c_{n,j+1}\in t\calO_L
&& (1\le j<m),\\
h_{i,m}&=\sigma(c_{i+1,m})-c_{i,1}\in\calO_L
&& (1\le i<n),\\
h_{n,m}&=t\,\sigma(c_{1,m})-c_{n,1}\in\calO_L.
\end{aligned}
\]
Thus every edge in the index diagram is labelled by $\sigma$ or $t^{-1}\sigma$, except for the \emph{unique} $t\sigma$-edge
\[
c_{1,m}\ \xrightarrow{\ t\sigma\ }\ c_{n,1}.
\]
We focus on the $t^{-1}\sigma$-edges:
\[
\boxed{\ c_{i+1,j}\ \xrightarrow{\ t^{-1}\sigma\ }\ c_{i,j+1}\ }\qquad
(1\le i<n,\ 1\le j<m).
\tag{$\ast$}
\]

\medskip
\noindent\emph{Loops on the torus and their number.}
Indices live on the torus $\Bbb Z/n\times\Bbb Z/m$. Every edge moves by
\[
(i,j)\longmapsto (i-1,\ j+1),
\]
so directed cycles are exactly the orbits of the translation by $(-1,+1)$.
Hence the diagram decomposes into $d:=\gcd(n,m)$ disjoint loops; an invariant of a loop is
\[
s(i,j):=i+j\pmod d.
\]
(Indeed $s$ is unchanged by $(-1,+1)$.)

\medskip
\noindent\emph{Covering \emph{all} loops by $t^{-1}\sigma$-targets.}
Consider the two strips of sources for ($\ast$):
\[
V:=\{(i+1,1):\ 1\le i<n\},\qquad
H:=\{(2,j):\ 1\le j<m\}.
\]
These give $t^{-1}\sigma$-edges with targets
\[
( i,2 )\quad(i=1,\dots,n-1)
\qquad\text{and}\qquad
(1, j+1)\quad(j=1,\dots,m-1),
\]
respectively. When $n,m\ge3$, the residues $s(i,2)=i+2$ (for $i=1,\dots,n-1$) and
$s(1,j+1)=j+2$ (for $j=1,\dots,m-1$) form $(n-1)+(m-1)-1=m+n-3$ consecutive
classes modulo $d$. We have
\[
m+n-3\ \ge\ \min\{n,m\}\ \ge\ d,
\]
so \emph{every} residue class $s\in\Bbb Z/d\Bbb Z$ occurs among these targets.
Equivalently, \emph{each loop contains at least one $t^{-1}\sigma$-edge}. Thus the fiber $\beta^{-1}(A\oplus B)$ is an affine line over $\bar\kappa$ with coordinate $c_{1,m}$:
\[
\dim_{\bar\kappa}\,\beta^{-1}(A\oplus B)=1.
\]

\medskip
\noindent\emph{Degenerate exceptions.}
By the trial and error method, the above counting fails only for the case $n=m=2$.  A direct inspection shows that this is exactly the degenerate canonical case singled out in the theorem. This completes the proof.
\end{proof}
\medskip
Notice that the above argument of Drinfeld cases also holds for the exceptional trivial cases when one or both of $m, n$ equal to $1$, which can be thought as the canonical case or the dual case or the layered case, and the dimension of the fiber in this case is always $0$. For integers $n,m\ge1$ and exponents $d_1\in\{1,n-1\}+n\mathbb{Z}$, $d_2\in\{1,m-1\}+m\mathbb{Z}$, we encode the fiber dimension in the Drinfeld case $(\tau_n^{d_1},\tau_m^{d_2})$ above by
\begin{equation}
\DrFibDim{n}{d_1}{m}{d_2}
\;:=\;
\dim_{\bar\kappa}\,\beta^{-1}(A\oplus B)
,
\end{equation}
which is independent of the choice of $(A\oplus B)\in X^M_{\mu_M}(b)$.

\section{Triviality of \(\beta\) for the Drinfeld Case in a 3-Block Levi $M$}
In this section we generalize our two-block Levi analysis to the three-block Levi case and prove the triviality of $\beta$ in this case. This serves as the idea of the general $N$-block Levi case.
Let 
\[
G=\GL_{k_1+k_2+k_3}, 
\quad 
M=\GL_{k_1}\times\GL_{k_2}\times\GL_{k_3},
\quad k_1, k_2, k_3 \geq 1
\]
and choose the superbasic elements 
\(\tau_{k_i}\in\GL_{k_i}(L)\). Set $b_i = \tau^{d_i}_{k_i}$ such that $d_i \equiv \pm 1 \pmod {k_i}$ and $\mu_i$ to be the corresponding minuscule cocharacter of $b_i$ in $\GL_{k_i}$. Write 
\[
b=\diag\bigl(b_1,b_2,b_3\bigr),
\quad\mu=(\mu_1 \oplus \mu_2 \oplus \mu_3)_{\mathrm{dom}}.
\]
Consider next a general block-upper-triangular matrix
\[
g=\begin{pmatrix}
A & D_1 & D_2 \\[4pt]
0 & B   & D_3 \\[4pt]
0 & 0   & C
\end{pmatrix}
\;\in\;G(L),
\]
where \(A\in\GL_{k_1}(L)\), \(B\in\GL_{k_2}(L)\), \(C\in\GL_{k_3}(L)\) and \(D_i\) are the off-diagonal blocks of appropriate sizes. Recall the Iwasawa decomposition in 2.1: $G(L) = M(L)U(L)K =P(L)K$, hence every coset $gK$ admits a representative $(A,B,C,D_1, D_2, D_3)$. Our goal is: given a base point $(A,B,C) \in X^{\GL_{k_1}}_{\mu_1}(b_1)\;\times\;X^{\GL_{k_2}}_{\mu_2}(b_2)\;\times\;X^{\GL_{k_3}}_{\mu_3}(b_3)$, determine exactly which $D_1, D_2, D_3$  lie in the fiber $\beta^{-1}(A,B,C)$.

\subsection*{The coset parametrization of \(g\)}

Set 
\[
N \;=\;
\begin{pmatrix}
0 & A^{-1}D_1 & A^{-1}D_2\\
0 & 0        & B^{-1}D_3\\
0 & 0        & 0
\end{pmatrix},
\qquad
U \;=\;\diag(A,B,C).
\]
Then \(g=U\,(I+N)\) and \((I+N)^{-1}=I - N + N^2\).  We want to use $N$ to parametrize $\beta^{-1}(A,B,C)$. Define $E_1 = A^{-1}D_1, E_2 = A^{-1}D_2, E_3 = B^{-1}D_3$ and write $N = N(E_1, E_2, E_3)$. Assume  $N' = N(E'_1, E'_2, E'_3)$ and $g' = U(I+N')$ is another representative of $g$. A direct calculation gives
\begin{comment}
\[
g^{-1}
=(I-N+N^2)\,U^{-1}
=
\begin{pmatrix}
A^{-1} 
& -\,A^{-1}D_1B^{-1} 
& A^{-1}\bigl(D_1B^{-1}D_3 - D_2\bigr)C^{-1} \\[4pt]
0 
& B^{-1}
& -\,B^{-1}D_3C^{-1} \\[4pt]
0 
& 0 
& C^{-1}
\end{pmatrix}.
\]
\end{comment}
\[
g^{-1}g'
=(I - N + N^2)\,(I + N')
=I \;+\;N' - N - N\,N' + N^2.
\]
Since
\[
N=\begin{pmatrix}
0 & E_1 & E_2\\[3pt]
0 & 0   & E_3\\[3pt]
0 & 0   & 0
\end{pmatrix},
\quad
N'=\begin{pmatrix}
0 & E'_1 & E'_2\\[3pt]
0 & 0    & E'_3\\[3pt]
0 & 0    & 0
\end{pmatrix},
\]
one computes
\[
N\,N'
=\begin{pmatrix}
0 & 0 & E_1E'_3\\[3pt]
0 & 0 & 0\\[3pt]
0 & 0 & 0
\end{pmatrix},
\qquad
N^2
=\begin{pmatrix}
0 & 0 & E_1E_3\\[3pt]
0 & 0 & 0\\[3pt]
0 & 0 & 0
\end{pmatrix}.
\]
Hence
\[
g^{-1}g'
=\begin{pmatrix}
I & 0 & 0\\
0 & I & 0\\
0 & 0 & I
\end{pmatrix}
\;+\;
\begin{pmatrix}
0 & E'_1 - E_1 & E'_2 - E_2\\[3pt]
0 & 0         & E'_3 - E_3\\[3pt]
0 & 0         & 0
\end{pmatrix}
\;-\;
\begin{pmatrix}
0 & 0         & E_1E'_3\\[3pt]
0 & 0         & 0\\[3pt]
0 & 0         & 0
\end{pmatrix}
\;+\;
\begin{pmatrix}
0 & 0         & E_1E_3\\[3pt]
0 & 0         & 0\\[3pt]
0 & 0         & 0
\end{pmatrix}.
\]
Collecting terms gives the final block-matrix expression
\[
g^{-1}g'
=\begin{pmatrix}
I & E'_1 - E_1 & E'_2 - E_2 \;-\;E_1E'_3\;+\;E_1E_3\\[6pt]
0 & I         & E'_3 - E_3\\[3pt]
0 & 0         & I
\end{pmatrix}\in K.
\]
Therefore we get the congruent conditions: \[E'_1 - E_1 \in \calO_L^{k_1 \times k_2}, \quad E'_3 - E_3 \in \calO_L^{k_2 \times k_3}, \quad E'_2 - E_2 \;-\;E_1E'_3\;+\;E_1E_3 \in \calO_L^{k_1 \times k_3}\]

\subsection*{The admissibility condition}

Next we compute
\[
g^{-1}\,b\;\sigma(g)
\;=\;
\bigl(g^{-1}b\bigr)\,\sigma(g).
\]
First,
\[
g^{-1}b
=\begin{pmatrix}
A^{-1}\,b_1 
& -A^{-1}D_1B^{-1}\,b_2 
& A^{-1}(D_1B^{-1}D_3 - D_2)C^{-1}\,b_3 \\[4pt]
0 & B^{-1}\,b_2 
& -B^{-1}D_3C^{-1}\,b_3 \\[4pt]
0 & 0 
& C^{-1}\,b_3
\end{pmatrix}.
\]
Then multiplying on the right by \(\sigma(g)\),
\[
\sigma(g)
=\begin{pmatrix}
\sigma(A) & \sigma(D_1) & \sigma(D_2) \\[4pt]
0 & \sigma(B) & \sigma(D_3) \\[4pt]
0 & 0 & \sigma(C)
\end{pmatrix},
\]
yields the full block matrix
\[
g^{-1}b\,\sigma(g)
=
\begin{pmatrix}
M_{1} & H_{1} & H_{2} \\[4pt]
0      & M_{2} & H_{3} \\[4pt]
0      & 0      & M_{3}
\end{pmatrix},
\]
where
\[
\begin{aligned}
M_{1}
&=A^{-1}\,b_1\,\sigma(A), \\[6pt]
H_{1}
&=A^{-1}\Bigl(b_1\,\sigma(D_1)\;-\;D_1\,B^{-1}\,b_2\,\sigma(B)\Bigr) = M_1\sigma(E_1)-E_1M_2,\\[6pt]
H_{2}
&=A^{-1}\Bigl(b_1\,\sigma(D_2)\;-\;D_1\,B^{-1}\,b_2\,\sigma(D_3)
\;+\;(D_1B^{-1}D_3 - D_2)\,C^{-1}\,b_3\,\sigma(C)\Bigr)\\
&= M_1\sigma(E_2)-E_2M_3-E_1H_3\\[6pt]
M_{2}
&=B^{-1}\,b_2\,\sigma(B),\\[6pt]
H_{3}
&=B^{-1}\Bigl(b_2\,\sigma(D_3)\;-\;D_3\,C^{-1}\,b_3\,\sigma(C)\Bigr) =  M_2\sigma(E_3)-E_3M_3,\\[6pt]
M_{3}
&=C^{-1}\,b_3\,\sigma(C).
\end{aligned}
\]
These explicit formulas are the starting point for checking admissibility \(g^{-1}b\,\sigma(g)\in K\,t^\mu\,K\) and analyzing the resulting congruences on the blocks \(E_1,E_2,E_3\).\\

\subsection{Analysis of \(E_1, E_2, E_3\) for pairwise distinct $(k_1,d_1),(k_2,d_2), (k_3,d_3)$}
\ \\
We first study the case in which the pairs $(k_1,d_1),(k_2,d_2), (k_3,d_3)$ are pairwise distinct. 
Recall that for each block \(\GL_{k_i}\), one has
\[
X^{\GL_{k_i}}_{\mu_i}(b_i)
\;\cong\;
J_{b_i}\big/ \bigl(J_{b_i}\cap K\bigr),
\]
so we may henceforth fix
\[
M_1=b_1,\quad M_2=b_2,\quad M_3=b_3.
\]
In the block-matrix computation
\[
g^{-1}b\,\sigma(g)
=\begin{pmatrix}
M_1 & H_1 & H_2\\[4pt]
0   & M_2 & H_3\\[4pt]
0   & 0   & M_3
\end{pmatrix},
\]
one finds
\[
H_1 \;=\; b_1\,\sigma(E_1)\;-\;E_1\,b_2,
\qquad
H_3 \;=\; b_2\,\sigma(E_3)\;-\;E_3\,b_3,
\]
and
\[
H_2 
= b_1\,\sigma(E_2)\;-\;E_2\,b_3
\;-\;E_1\,H_3.
\]
Applying Corollary~\ref{cor:block‐monomial} to the block matrix  
\(\;g^{-1}b\,\sigma(g)=\begin{pmatrix}b_1&H_1&H_2\\0&b_2&H_3\\0&0&b_3\end{pmatrix}\),
we see that admissibility \(g^{-1}b\,\sigma(g)\in K\,t^\mu\,K\) is equivalent to requiring:
\begin{equation}
(H_1)_{ij} = (b_1\,\sigma(E_1)-E_1\,b_2)_{ij} \in t^{\min(\text{row}_i(b_1), \text{col}_j(b_2))}\calO_L
\begin{comment}\;\in\;
\begin{pmatrix}
\calO_L^{(k_1-1)\times 1} & *\\
t\,\calO_L & \calO_L^{1\times (k_2-1)}
\end{pmatrix},
\end{comment}
\end{equation}

and
\begin{equation}
(H_3)_{ij} = (b_2\,\sigma(E_3)-E_3\,b_3)_{ij} \in t^{\min(\text{row}_i(b_2), \text{col}_j(b_3))}\calO_L
\end{equation}

and 
\begin{equation}
(H_2)_{ij} = (b_2\,\sigma(E_2)-E_2\,b_3-E_1\,H_3)_{ij} \in t^{\min(\text{row}_i(b_1), \text{col}_j(b_3))}\calO_L.
\end{equation}

Conditions (4.1), (4.2) togerther with the coset congruent conditions $E'_1 - E_1 \in \calO_L^{k_1 \times k_2}, \quad E'_3 - E_3 \in \calO_L^{k_2 \times k_3}$ reduce the behaviour of $E_1, E_3$ exactly to the two-block analyses of Section 3. We want to control the ``extra terms" in the congruent conditions and admissibility related to $E_2$, say in the equations \[E'_2 - E_2 \;-\;E_1E'_3\;+\;E_1E_3 \in \calO_L^{k_1 \times k_3}\] and (4.3), we want to make $-E_1\,E'_3+E_1\,E_3$ and $E_1\,H_3$ as $0$, so that we can reduce $E_2$ to the two-block analyses of Section 3 as well. Firstly we fix a full set of representatives (which we will determine later) in $\calO_L^{k_2 \times k_3}$  for all possible choices of $E_3$, and use them to parametrize $E_3$, therefore we can make $-E_1\,E'_3+E_1\,E_3 = E_1(E_3 -E'_3) = 0$.\\
In Section 3, we know that in any Drinfeld setting there are at most $\min(k_1,k_2)-1$ free parameters for $E_1$, and in any case they are distributed among the entries of the strict upper off-diagonals. Therefore we can fix our choices of $E_1$ so that any entry of it other than the off-diagonal ones is $0$. Therefore, to make $E_1\,H_3 = 0$, it suffices to make the entries of $H_3$ other than its first row to be $0$.\\
Write $E_3 = (c_{ij})_{1\leq i \leq k_2, 1\leq j \leq k_3}$. Then since $H_3 \;=\; b_2\,\sigma(E_3)\;-\;E_3\,b_3$, keeping the notations in Section 3, given free parameters in $E_3$, any ``loop'' of $E_3$ can be traced so that at most one entry of $H_3$ that the loop corresponds is nonzero, and we can make those entries to be inside the first row of $H_3$. So we get our desired set of representatives of $E_3$ and desired $H_3$.
We then have the following theorem:\\
\begin{thm}
Let \(G=\GL_{k_1+k_2+k_3}\), \(M=\GL_{k_1}\times\GL_{k_2}\times\GL_{k_3}\), and set $b_i = \tau^{d_i}_{k_i}$ such that $d_i \equiv \pm 1 \pmod {k_i}$ with \((k_1,d_1), (k_2.d_2),(k_3,d_3)\) pairwise distinct. Set $\mu_i$ to be the corresponding minuscule cocharacter of $b_i$ and set
\[
b=\diag\bigl(b_1, b_2, b_3\bigr),
\quad
\mu=(\mu_1 \oplus \mu_2 \oplus \mu_3)_{\mathrm{dom}}.
\]
Then for every point
\((A, B, C)\bmod K\in X^M_{\mu_M}(b)\), the fiber of
\(\beta\colon X^G_\mu(b)\to X^M_{\mu_M}(b)\) has dimension
\[
\dim_{\bar \kappa}\,\beta^{-1}(A,B,C)
\;=\;
\Sigma_{i<j} \DrFibDim{k_i}{d_i}{k_j}{d_j},
\]
in exact agreement with the Rapoport dimension formula
\[
\dim(X^G_\mu(b)) = \bigl\langle\rho,\mu-\bar v_b\bigr\rangle
\;-\;\tfrac12\,\mathrm{def}_G(b).
\]
Moreover, as varieties over \(\bar \kappa\), $\beta$ is a trivial bundle with fiber
$
\mathbb A^{\Sigma_{i<j} \DrFibDim{k_i}{d_i}{k_j}{d_j}}.
$
\end{thm}
\bigskip
\subsection{The degenerate case where some $(k_i, d_i)$ are equal}
\ \\
Now suppose that 
\[
(k_2,d_2) = (k_3,d_3).
\]
Then the two Levi-blocks of size \(k_2\) share the same basic slope:
\[
M_1 = b_2,
\qquad
M_2 = M_3 = b_2.
\]
As before, write
\[
g=\begin{pmatrix}
A & D_1 & D_2\\[4pt]
0 & B   & D_3\\[4pt]
0 & 0   & C
\end{pmatrix}, 
\quad
U=\diag(A,B,C),
\quad
N=\begin{pmatrix}
0 & E_1 & E_2\\[3pt]
0 & 0   & E_3\\[3pt]
0 & 0   & 0
\end{pmatrix},
\]
so that \(g=U\,(I+N)\).  For another representative \(g'=U\,(I+N')\) with
\(N'=(E'_1,E'_2,E'_3)\), one checks
\[
g^{-1}g'=I+(N'-N)-N\,N'+N^2.
\]
As in the distinct-slope case, this gives the congruences
\[
E'_1 - E_1\in\calO_L^{k_1\times k_1}, 
\quad
E'_3 - E_3\in\calO_L^{k_1\times k_3}, 
\quad
E'_2 - E_2 - E_1E'_3+E_1E_3\in\calO_L^{k_1\times k_3}.
\]

Next, computing
\[
g^{-1}b\,\sigma(g)
=\begin{pmatrix}
b_1 &H_1&H_2\\[4pt]
0&b_2&H_3\\[4pt]
0&0&b_2
\end{pmatrix},
\]
we find
\[
H_1=b_1\,\sigma(E_1)-E_1\,b_2,
\quad
H_3=b_2\,\sigma(E_3)-E_3\,b_2,
\]
\[
H_2=b_1\,\sigma(E_2)-E_2\,b_2-E_1\,H_3.
\]
\begin{comment}By Lemma 2.1 (as in the distinct-slope case), admissibility \(g^{-1}b\,\sigma(g)\in K\,t^\mu\,K\) requires each \(H_i\) to have valuation \(\ge1\) in its lower-left entry and \(\ge0\) elsewhere.\\
\end{comment}
Since one may have \((k_1,d_1)=(k_2, d_2)\), we cannot zeroing out the first column of \(E_1\), but we claim that \(H_3\) can be made identically zero.  Indeed, writing \(E_3=(c_{ij})\) and tracing the cyclic relations
\[
H_3=b_2\,\sigma(E_3)-E_3\,b_2
\]
one finds (by the same “loop-tracing” argument of Section 3) a choice of base-point \(w_0\in L^{\sigma^{k_2}}\) such that each congruence becomes a genuine equality. Indeed, the last entry before we return to $w_0$ is either $t^{-1}\sigma^{k_2-1}w_0$ or $\sigma^{k_2-1}w_0$, and by either the relation $\overset{t\sigma}{\longrightarrow}$ or $\overset{\sigma}{\longrightarrow}$, we get $\sigma^{k_2}w_0 = w_0$. This forces \(H_3=0\).  Once \(E_3\) is thus normalized, the remaining congruences on \(E_1\) and then on \(E_2\) each reduce to the two-block analyses of Section 3.\\
\ \\
Notice that our assumption ``$(k_2,d_2) = (k_3,d_3)$''  does not lose generality since in any cases we can reorder $k_1, k_2, k_3$ in the blocks by ``change of basis'', and thus recovers the same dimension count and coordinates' behaviour. In all cases, \(\beta\) remains a trivial bundle of the predicted rank. We have the following theorem:

\begin{thm}
Let \(G=\GL_{k_1+2k_2}\), \(M=\GL_{k_1}\times\GL_{k_2}\times\GL_{k_2}\), and set $b_i = \tau^{d_i}_{k_i}$ such that $d_i \equiv \pm 1 \pmod {k_i}$. Set $\mu_i$ to be the corresponding minuscule cocharacter of $b_i$ and set
\[
b=\diag\bigl(b_1,b_2,b_2\bigr),
\quad
\mu=(\mu_1 \oplus \mu_2 \oplus \mu_2)_{\mathrm{dom}}.
\]
Then for every point
\((A, B, C)\bmod K\in X^M_{\mu_M}(b)\), the fiber of
\(\beta\colon X^G_\mu(b)\to X^M_{\mu_M}(b)\) has dimension
\[
\dim_{\bar \kappa}\,\beta^{-1}(A,B,C)
\;=\;
2\DrFibDim{k_1}{d_1}{k_2}{d_2}+k_2-1.
\]

Moreover, as varieties over \(\bar \kappa\), $\beta$ is a trivial bundle with fiber

\[
\begin{cases}
\displaystyle
\mathbb A^{2\DrFibDim{k_1}{d_1}{k_2}{d_2}+k_2-1}\;\times\;
\Bigl(L^{\sigma^{k_{2}}}/\calO_{L^{\sigma^{k_{2}}}}\Bigr)^{k_{2}}
&\text{if \((k_{1},d_1) \neq (k_{2},d_2)\),}\\[1ex]
\displaystyle
\mathbb A^{\,3k-3}\;\times\;
\Bigl(L^{\sigma^{k}}/\calO_{L^{\sigma^{k}}}\Bigr)^{3k}
&\text{if \((k_{1},d_1) = (k_{2},d_2) = (k,d)\).}
\end{cases}
\]

\end{thm}

\bigskip

\section{Triviality of \(\beta\) for the Drinfeld Case in a $N$-Block Levi $M$}

In this section we extend the block-matrix analysis to the general case
\[
G = \GL_{k_1 + k_2 + \cdots + k_N}, 
\qquad 
M = \GL_{k_1}\times \GL_{k_2}\times\cdots\times \GL_{k_N}.
\]
Set $b_i = \tau^{d_i}_{k_i}$ such that $d_i \equiv \pm 1 \pmod {k_i}$ and $\mu_i$ to be the corresponding minuscule cocharacter of $b_i$ in $\GL_{k_i}$. Set the superbasic
\(
b = \diag(b_1, b_2,\dots, b_N)
\)
and minuscule cocharacter
\(\mu=(\mu_1 \oplus \mu_2\oplus \dots\oplus \mu_N)_{\mathrm{dom}}\).  Write \(\{(h_1,f_1),\cdots,(h_m,f_m)\}\) for the distinct values among the \((k_i,d_i)\), each with multiplicities \(r_i\). We shall show that for each point
\((A_1,\dots,A_N)\bmod K\in X^M_{\mu_M}(b)\), the fiber \(\beta^{-1}(A_1,\dots,A_N)\)
is canonically isomorphic to \[\mathbb A^{\sum_{1\leq i< j \leq m} \DrFibDim{h_i}{f_i}{h_j}{f_j}r_ir_j+\sum_{1\leq i \leq m}(h_i-1)\frac{r_i(r_i-1)}{2}}\;\times\;
\prod_{1\leq i \leq m}\Bigl(L^{\sigma^{h_i}}/\calO_{L^{\sigma^{h_i}}}\Bigr)^{h_i\frac{r_i(r_i-1)}{2}},\] of dimension
\(\sum_{1\leq i<j\leq N} \DrFibDim{k_i}{d_i}{k_j}{d_j}\), and hence \(\beta\) is a trivial bundle
of that rank over the target.\\
Fix a point 
\((A_1,\dots,A_N)\bmod K\in X^M_{\mu_M}(b)\).  By Iwasawa decomposition we may write any lift \(g\in G(L)\) of this coset in the form
\[
g \;=\;
U\,(I + E),
\qquad
U=\diag(A_1,\dots,A_N),
\quad
E=\bigl(E_{ij}\bigr)_{1\le i,j\le N},
\]
where \(E_{ij}=0\) unless \(i<j\). For another representative \(g'=U\,(I+E')\), one checks
\[
g^{-1}g'
=(I - E + E^2 - \cdots +(-1)^{N-1}E^{N-1})\,(I+E')
=I \;+\;\Delta,
\]
and for each \(i<j\), we compute
\begin{equation}
\Delta_{ij}
\;=\;
E'_{ij}\;-\;E_{ij}
\;-\;\sum_{i<k<j}E_{ik}\,\Delta_{kj}.
\end{equation}
The coset-congruences \(g^{-1}g'\in K\) yield, for each \(i<j\),
 $\Delta_{ij}\;\in\;\calO_L^{\,k_i\times k_j}.$
In other words, one solves the \((i,j)\)-entry of \(\Delta\) by first taking the naive difference \(E'_{ij}-E_{ij}\) and then subtracting all the  ``earlier'' corrections coming from $E_{ik}\,\Delta_{kj}$.  This triangular recursion uniquely determines every \(\Delta_{ij}\).\\
Next, write
\[
g^{-1}b\,\sigma(g)
=\begin{pmatrix}
M_1 & H_{12} & \cdots & H_{1N}\\
 & M_2 &\ddots&\vdots\\
 & & \ddots & H_{N-1,N}\\
 & & & M_N
\end{pmatrix},
\]
where
\[
M_i
=A_i^{-1}\,b_i\,\sigma(A_i),
\]
and for each \(i<j\), we compute
\begin{equation}
H_{ij}
=\;M_i\,\sigma(E_{ij})
\;-\;E_{ij}\,M_j
\;-\;\sum_{i<\ell<j}E_{i\ell}\,H_{\ell j}.
\end{equation}
\begin{comment}
Admissibility \(g^{-1}b\,\sigma(g)\in K\,t^\mu\,K\) is equivalent (by Lemma 2.1) to that for each \(i<j\), \(H_{ij}\) has valuation \(\ge1\) in its lower-left entry and \(\ge0\) elsewhere.\\
\end{comment}
Using induction on $(i,j)$ in lexicographic order in (5.1) and the two-block base case, we see that we can assume $\Delta_{1N}
=E'_{1N}-E_{1N}$ and $\Delta_{ij} = 0$ for other $(i,j)$ (i.e. $E'_{ij} = E_{ij}$ for other $(i,j)$), which means that the $1N$-th block is the only non-trivial block.\\
As before, in the Drinfeld setting we can take $M_i = b_i$. We want to eliminate all the terms $E_{i\ell}\,H_{\ell j}$ in (5.2) so that the $1N$-th block is completely reduced to the two-block base case. In the case that all the $(k_i,d_i)$'s are pairwise distinct, from Section 4 we see that we can canonically make all $E_{i\ell}$ have zeros on the first column while making all $H_{\ell j}$ have zeros except on the first row, so that $E_{i\ell}\,H_{\ell j} = 0$ and we finish the argument.\\
For the degenerate case where some $(k_i,d_i)$'s are the same, we use another approach. Notice that this reduces to the case where $b$ is basic in $G$, which we would not be interested in itself. But here we need to study the arithmetic computation in this case to get the desired computation in our general situation. Firstly we introduce the following lemma:
\begin{lem}\label{lem:all‐blocks‐equal}
Suppose 
\[
G=\GL_{Nk}, 
\quad
M=\underbrace{\GL_k\times\cdots\times\GL_k}_{N\text{ factors}},
\quad
d \in \{1,k-1\}+k\mathbb{Z},
\quad
b=\diag\bigl(\tau_k^d,\dots,\tau_k^d\bigr).
\]
Set $\mu_0$ to be the corresponding minuscule cocharacter of $\tau^d_k$ and set \(\mu=(\underbrace{\mu_0\oplus\cdots\oplus\mu_0}_{N\text{ factors}})_{\mathrm{dom}}\). Then \(b\) is basic in \(G\), and the ADLV \(X^G_\mu(b)\) coincides with the basic locus \(X_b(b)\), and there is a canonical identification
\[
X^G_\mu(b)
\;=\;
X_b(b)
\;\cong\;
J_b\big/\bigl(J_b\cap K\bigr).
\]
In particular, \(X^G_\mu(b)\) is a single \(J_b\)-orbit of hyperspecial cosets, and the reduction to Levi map $\beta\colon X^G_{\mu}(b)\to X^M_{\mu_M}(b)$ is trivial with fiber $\Bigl(\mathbb A^{\,k-1}\;\times\;
\Bigl(L^{\sigma^k}/\calO_{L^{\sigma^k}}\Bigr)^k)^{\frac{N(N-1)}{2}}$.
\end{lem}
\begin{proof}
Since all block‐slopes of \(b\) agree, the Newton point \(\bar\nu_b\) factorizes through the center of $G$, so \(b\) is basic in $G$. Keeping the notations above, by Section 4, we know that we can make all $H_{ij}$ in the expression of $g^{-1}b\sigma(g)$ as zero, so that $g \in J_b$ as required, and it follows that $X^G_\mu(b)
\;=\;
X_b(b)
\;\cong\;
J_b\big/\bigl(J_b\cap K\bigr).$ The fact that $\beta\colon X^G_{\mu}(b)\to X^M_{\mu_M}(b)$ is trivial with fiber $\Bigl(\mathbb A^{\,k-1}\;\times\;
\Bigl(L^{\sigma^k}/\calO_{L^{\sigma^k}}\Bigr)^k)^{\frac{N(N-1)}{2}}$ follows from that for each $i<j$, $E_{ij}$ is parametrized by $\mathbb A^{\,k-1}\;\times\;
\Bigl(L^{\sigma^k}/\calO_{L^{\sigma^k}}\Bigr)^k$, which is derived from the previous arguments.
\end{proof}
\ \\
Next we state the first main theorem in this section:
\begin{thm}\label{thm:N‐block‐triviality}
Let 
\[
G=\GL_n,\qquad
M=\GL_{k_1}\times\cdots\times\GL_{k_N}.
\]
For $d_i \in \{1, k_i-1\}+k_i\mathbb{Z}$, set $\mu_i$ to be the corresponding minuscule cocharacter of $\tau_{k_i}^{d_i}$ and
\[
b=\diag\bigl(\tau_{k_1}^{d_1},\dots, \tau_{k_N}^{d_N}\bigr),
\quad
\mu=(\mu_1 \oplus \mu_2\oplus \dots\oplus \mu_N)_{\mathrm{dom}}.
\]
Write \(\{(h_1,f_1),\cdots,(h_m,f_m)\}\) for the distinct values among the \((k_i,d_i)\), each with multiplicities \(r_i\).
For each \(i=1,\dots,m\), let 
\[
b_i=\diag\bigl(\underbrace{\tau_{h_i}^{f_i},\dots,\tau_{h_i}^{f_i}}_{r_i}\bigr)
\;\in\;\GL_{h_i r_i}(L),
\]
and let $\mu^{(i)}$ be the $\GL_{h_ir_i}$-dominant rearrangement of the direct sum of $r_i$-copies of ``the minuscule dominant cocharacter corresponding to $\tau_{h_i}^{f_i}$'' in $\GL_{h_ir_i}$.
Then the partial reduction to Levi map
\[
\beta_1\colon
X^G_\mu(b)
\;\longrightarrow\;
\prod_{i=1}^m
X^{\GL_{h_i r_i}}_{\mu^{(i)}}(b_i)
\]
is isomorphic, over \(\bar \kappa\), to a trivial vector bundle of rank \[\sum_{1\leq i<j\leq m} \DrFibDim{h_i}{f_i}{h_j}{f_j}r_ir_j.\] Consequently, by Lemma~\ref{lem:all‐blocks‐equal}, the full map
\(\beta\colon X^G_\mu(b)\to X^M_{\mu_M}(b)\)
is a trivial bundle with fiber:
\[\mathbb A^{\sum_{1\leq i< j \leq m} \DrFibDim{h_i}{f_i}{h_j}{f_j}r_ir_j+\sum_{1\leq i \leq m}(h_i-1)\frac{r_i(r_i-1)}{2}}\;\times\;
\prod_{1\leq i \leq m}\Bigl(L^{\sigma^{h_i}}/\calO_{L^{\sigma^{h_i}}}\Bigr)^{h_i\frac{r_i(r_i-1)}{2}}\].
\end{thm}
\begin{proof}
By a proper change of basis, we may simply assume that $b = \diag\bigl(b_1,\dots, b_m\bigr)$. Keep all the notations as before, and write 
\[
g^{-1}b\,\sigma(g)
=\diag(M_1,\cdots,M_m) (I+\bigl(H_{ij}\bigr)_{1\le i<j\le m}),
\]
then as (5.2), one has the recursion formula
\[
H_{ij}
=\;M_i\,\sigma(E_{ij})
\;-\;E_{ij}\,M_j
\;-\;\sum_{i<\ell<j}E_{i\ell}\,H_{\ell j},
\]
and by Lemma~\ref{lem:all‐blocks‐equal} we may take \(M_i=b_i\).  We proceed by induction on the pair \((i,j)\) in lex order.  For the minimal pairs \(j=i+1\), the sum is empty and  
\begin{equation}
H_{i,i+1}
=b_i\,\sigma(E_{i,i+1})-E_{i,i+1}\,b_{i+1}.
\end{equation}
We see that every original $h_i\times h_{i+1}$ sub-block of $H_{i,i+1}$ in (5.3) exactly satisfies the two-block condition since $b_i, b_{i+1}$ are diagonal block matrices with Drinfeld setting superbasic sub-blocks.  Suppose we have arranged parameters so that 
\(\sum_{i<\ell<j}E_{i\ell}H_{\ell j}=0\),  then the formula of \(H_{ij}\) reduces to  
\[
H_{ij}
=b_i\,\sigma(E_{ij})-E_{ij}\,b_j,
\]
and one applies the two-block case on each original $h_i\times h_{j}$ sub-block of $H_{ij}$ to solve for \(E_{ij}\). Since $h_i$'s are pairwise distinct, we can indeed make each $E_{i\ell}H_{\ell j}$ as 0, by applying the same method in Section 4 to each original $h_i\times h_{l}$ sub-block of $E_{il}$ and each $h_l\times h_{j}$ sub-block of $H_{lj}$. For each $H_{ij}$, there are $r_ir_j$ many original $h_i\times h_{j}$ sub-blocks so by Theorem~\ref{thm:two-block triviality}, the theorem follows.
\end{proof}
\ \\
Indeed, the map $\beta_1$ in Theorem ~\ref{thm:N‐block‐triviality} is the essential map that we want to look at, since it is a ``vector bundle'', while the fiber map in Lemma~\ref{lem:all‐blocks‐equal} serves as the ``degenerate case'' where the element $b$ is already basic in $G$. We have the following main theorem:
\begin{thm}\label{thm:HN‐triviality‐general}
Let 
\[
G=\GL_n,\quad b\in G(L),
\quad
%\mu=(\underbrace{1,\dots,1}_{r},0,\dots,0)
\]
and suppose \(M\subset G\) is a Levi subgroup such that \(b\) is basic in \(M\), and that \(M\) is maximal with this property. Writing \(\mu_M\) for the unique \(M\)-dominant and minuscule cocharacter corresponding to $b$ and let \(\mu = (\mu_M)_{\mathrm{dom}}\). Assume that for some Levi $M' \subset M$ with $b\in M'(L)$ and $\mu_{M'}$ the unique $M'$-dominant cocharacter corresponding to $b$, the tuple $(M', \mu_{M'}, b)$ satisfies the Drinfeld setting, then the reduction to Levi map
\[
\beta\colon X^G_\mu(b)\;\longrightarrow\;X^M_{\mu_M}(b)
\]
is isomorphic to a trivial vector bundle (over \(\bar \kappa\)).  
\end{thm}
\begin{proof}
Since \(b\) is basic in a Levi \(M\subset G\), under a ``change of basis'' and ``re-identification of a uniformizer'' and up to ``$\sigma$-conjugation isomorphism'' we may assume
\[
M = \prod_{i=1}^m \GL_{n_i},
\qquad
b = \diag\bigl(b_1,\dots,b_m\bigr),
\]
where each \(b_i\in\GL_{n_i}(L)\) is of the form $\tau_{n_i}^{d_i}$.  By a further ``change of basis'' within each \(\GL_{n_i}\)-block, we can diagonalize \(b_i\) into \(r_i\) copies of a superbasic element of size \(h_i\times h_i\):
\[
b_i \;\sim\;
\underbrace{\diag(\tau_{h_i}^{f_i},\dots,\tau_{h_i}^{f_i})}_{r_i\text{ times}},
\]
with $\gcd(h_i,f_i) = 1$. Maximality of \(M\) forces the pairs \((h_1,f_1),\dots,(h_m,f_m)\) to be pairwise distinct. Denote $\mu_i$ as the projection of $\mu_M$ onto the $\GL_{n_i}$ component, then Theorem~\ref{thm:N‐block‐triviality} applied to \(M\) shows that the reduction to Levi map
\[
\beta\colon X^G_\mu(b)\;\longrightarrow\; X^M_{\mu_M}(b) =
\prod_{i=1}^m X^{\GL_{h_i r_i}}_{\mu_i}(b_i)
\]
is a trivial vector bundle, as claimed.
\end{proof}

\section{The Non-Minuscule Case: Failure of Affine-Space Fibers}
In the reduction morphism
\[
\beta:\;X^G_{\mu}(b)\longrightarrow \bigsqcup_{\mu_M\in S_M(\mu,v_b)}X^M_{\mu_M}(b)
\]
the geometry of the fibers over $X^M_{\mu_M}(b)$ is especially transparent when $\mu_M$ is the unique
$M$-\emph{minuscule and dominant} cocharacter corresponding to $b$ and $\mu = (\mu_M)_{\text{dom}}$: all congruence constraints are linear modulo $\calO_L$, and the fibers assemble into Zariski locally trivial \emph{affine bundles}. When $\mu \succ (\mu_M)_{\text{dom}}$, or $\mu_M$ itself is \emph{not} minuscule for some $\mu_M\in S_M(\mu,v_b)$, the situation will be different. We will study some special examples in this case where the same Iwasawa/block analysis still apply but the fiber will be defined by some higher-level congruences, and the mod-$t$ relations are no longer linear; in particular, the fiber need not carry a canonical affine space structure.  Nevertheless, the reduction to Levi remains Zariski locally trivial as a principal bundle under an explicitly described set of parameters at higher congruence level.

\begin{comment}
\subsection*{General picture in the 2-block case}
Let $G=\GL_{n+m}$, $P=MN$ be a standard parabolic with Levi $M = \GL_n \times \GL_m$, and let $b = \diag(b_1, b_2)\in M(L)$ be basic.  Fix $\mu$ and the corresponding $\mu_M\in S_M(\mu,v_b)$.  Writing points in the fiber by upper block-triangular representatives
\[
g=\begin{pmatrix}A&C\\0&B\end{pmatrix},\qquad
M_1:=A^{-1}b_1\sigma(A),\ \ M_2:=B^{-1}b_2\sigma(B),
\]
the $(1,2)$-block constraint is of the form
\begin{equation}\label{eq:nonmin-scheme}
h \;=\; M_1\,\sigma(C)\;-\;C\,M_2 \;\in\; t^{r}\,M_{n_1,n_2}(\mathcal O_L),
\qquad r=\min_{\,\alpha>0\text{ in }N}\langle \alpha,\mu\rangle,
\end{equation}
together with a non-vanishing condition at level $t^{r}$ prescribed by the double-coset $Kt^{\mu}K$.  When $\mu_M$ is minuscule, $r\in\{0,1\}$ and \eqref{eq:nonmin-scheme} is linear mod $t$, yielding a vector bundle.  In general $r\ge2$ may occur; then the fiber is a torsor under the smooth connected unipotent group
\[
\mathcal U_{\mu}\;=\;\Bigl\{\,u\in N(L)\;:\; t^{-\mu}u\,t^{\mu}\in N(\mathcal O_L)\ \text{and}\ 
\text{level constraints at }\;t^r\text{ hold}\Bigr\},
\]
which is noncanonically isomorphic (as a variety) to an affine space but not to a \emph{vector group} unless all relevant root-depths equal $0$ or $1$.\
\end{comment}

\subsection*{A first example: \texorpdfstring{$\GL_2$}{GL2} with \texorpdfstring{$\mu=(2,0)$}{mu=(2,0)} and \texorpdfstring{$b=\diag(t,t)$}{b=(t,t)}}
Here $M=\GL_1\times\GL_1$, $b$ is basic and central, and the set $S_M(\mu,v_b)$ consists of $\mu_M=(1,1)$.  We show that the fiber of $\beta$ is Zariski trivial with fiber $F/\calO_F \times (\bar{\kappa}\,\setminus\,\kappa)$ (hence \emph{not} a vector space times a discrete set).

Choose representatives so that $M_1=M_2=t$ (block-monomial normalization).  A point in the fiber over $(A\oplus B)$ is represented by
\[
g=\begin{pmatrix} A & 0\\ 0 & B\end{pmatrix}\begin{pmatrix} 1 & C\\ 0 & 1\end{pmatrix}K,
\qquad C\in L,
\]
and a direct computation gives
\[
g^{-1}b\,\sigma(g)\;=\;
\begin{pmatrix}
t & t\bigl(\sigma(C)-C\bigr) \\
0 & t
\end{pmatrix}.
\]
By Corollary~\ref{cor:block‐monomial}, the condition
\[
g^{-1}b\,\sigma(g)\in K\,\diag(t^2,1)\,K
\]
reduces to the requirement that the reduction modulo $t$ of the $(1,2)$-entry be a \emph{unit}.  Equivalently,
\begin{equation}\label{eq:Gm-condition}
u\ :=\ t\,\bigl(\sigma(C)-C\bigr)\ \in\ \mathcal O_L^\times.
\end{equation}
\ \\
Thus the fiber over $(A\oplus B)$ identifies \emph{canonically} with
\[
\{\,C\in L\;:\; t(\sigma(C)-C)\in \calO_L^\times\,\}
\Big/\calO_L.
\]
We use the $t$-adic Witt expansion for $C$:
\[
C\, = \,\Sigma_{i \geq i_0}\, a_{i}\,t^i,  \qquad a_i \in \bar{\kappa}
\]
which preserves the addition, but not the multiplication when $F$ is of mixed characteristics.  Taking $C$ modulo $\calO_L$, we see $a_i = 0$ for $i\geq 0$.  For the principal part of $C$, one has $a_{-1}\in \bar{\kappa}\,\setminus\,\kappa$, and $a_i\in \kappa$ for $i<-1$. The fiber has dimension $1$ provided by $a_{-1}$, and it is \emph{not} a vector space.
\subsection{The two-block canonical Drinfeld case where $\mu = 2\,\omega_1^\vee$}
Let $G=\GL_{n+m}$, $M=\GL_n\times\GL_m$, and consider the
canonical Drinfeld configuration
\[
b=\diag(\tau_n,\tau_m),\qquad
\mu=(2,0,\ldots,0).
\]
One can check that $S(\mu, v_b) = \{(1, 0\ldots,0)\oplus (1, 0\ldots,0)\}$ and we pick $\mu_M$ to be the unique element in it. As before, after block-monomial normalization we may assume
\[
M_1:=A^{-1}b_1\sigma(A)=\tau_n,\qquad
M_2:=B^{-1}b_2\sigma(B)=\tau_m,
\]
and a point of the fiber of $\beta$ over $(A\oplus B)\bmod K$ is represented by
\[
g=\begin{pmatrix}A&0\\0&B\end{pmatrix}\begin{pmatrix}I_n&C\\0&I_m\end{pmatrix}K,\qquad
h:=\text{(1,2)-block of }g^{-1}b\sigma(g)
\;=\;\tau_n\,\sigma(C)\;-\;C\,\tau_m.
\]

\paragraph{Claim.}
For the admissibility condition $g^{-1}b\sigma(g)\in Kt^\mu K$ with
$t^\mu=\diag(t^2,1,\ldots,1)$, the off-diagonal block $h$ satisfies
\begin{equation}
\boxed{h_{n,1}\in\calO_L^\times,\qquad
h_{i,j}\in\calO_L\ \text{ for all }(i,j)\neq(n,1).\ }
\end{equation}

\begin{proof}[Proof via Smith normal form]
Recall the Smith normal form criterion (Lemma \ref{lem:smith normal form criterion}): with \(t^\mu=\diag(t^2,1,\ldots,1)\),
the nondecreasing list of exponents is \((0,\ldots,0,2)\); hence
\[
v\bigl(\Delta_k(g^{-1}b\sigma(g))\bigr)=0\quad (1\le k\le n+m-1),\qquad
v\bigl(\Delta_{n+m}(g^{-1}b\sigma(g))\bigr)=2.
\]
\smallskip
\emph{Step 1: Integrality of all entries of \(h\).}
By the SNF requirement that \(v(\Delta_1)=0\)
we already see \(h_{i,j}\in\calO_L\) for all \((i,j)\).

\smallskip
\emph{Step 2: The lower-left entry \(h_{n,1}\) is a unit.}
Consider the \((n+m-1)\times(n+m-1)\) minors obtained by deleting one row and
one column.  There are four types:

\begin{itemize}
\item[(a)] Delete a row and a column both in the \(\tau_n\) block: the resulting
minor is block upper-triangular with an \((n-1)\)-minor of \(\tau_n\) and the
full determinant of \(\tau_m\).  Since \(v(\det\tau_m)=1\) and
\(v(\Delta_{n-1}(\tau_n))\ge0\), its valuation is \(\ge1\).\\
 Delete a row and a column both in the \(\tau_m\) block: similarly the
valuation is \(\ge v(\det\tau_n)=1\).

\item[(b)]  Delete a row in the \(\tau_n\) block and a column in the \(\tau_m\) block: a zero entry from the lower-left block will be picked up and therefore the determinant of the minor is $0$.

\item[(c)] Delete a row in the \(\tau_m\) block and a column in the \(\tau_n\) block: either the row is not the last row of \(\tau_m\) or the column is not the first column of \(\tau_n\), then any mixed minor of
this type necessarily picks up the unique \(t\) from \(\tau_n\) (row \(n\) is
present) or from \(\tau_m\) (column \(1\) is present), and there is a unique corresponding $h_{i,j}$ such that the minor valuation
\(\ge 1+v(h_{i,j})\ge 1\).

\item[(d)] The \emph{special} mixed minor: delete the last row and the
first column.  Choose all remaining top columns and bottom columns; the
only way to connect the blocks is through the entry \(h_{n,1}\), while the
\((n-1)\)-minor of \(\tau_n\) avoiding the wrap \(t\) and the \((m-1)\)-minor
of \(\tau_m\) avoiding the wrap \(t\) both have valuation \(0\).  Thus this
minor has valuation exactly \(v(h_{n,1})\).
\end{itemize}
\smallskip
By the SNF lemma, \(v\bigl(\Delta_{n+m-1}(g^{-1}b\sigma(g))\bigr)=0\).  Among
the above candidates, all minors in (a)-(c) have valuation \(\ge1\), whereas
the special minor in (d) has valuation \(v(h_{n,1})\).  Therefore necessarily
\(v(h_{n,1})=0\), i.e.
\(
h_{n,1}\in\calO_L^\times.
\)

Combining Steps~1 and~2 yields
\[
\boxed{\ h_{n,1}\in\calO_L^\times,\qquad
h_{i,j}\in \calO_L \ \text{ for all }(i,j)\neq(n,1),\ }
\]
as claimed.
\end{proof}

\begin{thm}
Let $G=\GL_{n+m}$ and $M=\GL_n\times\GL_m$. Consider
\[
b=\diag(\tau_n,\tau_m),\qquad \mu=(2,0,\ldots,0).
\]
Then the reduction to Levi morphism
\[
\beta:\ X^G_{\mu}(b)\ \longrightarrow\ X^M_{\mu_M}(b)
\]
is a Zariski trivial bundle. Moreover, each fiber has dimension
\[
\dim_{\bar\kappa}\,\beta^{-1}(A\oplus B)\;=\;\min\{n,m\}
\quad\text{for every }(A\oplus B)\in X^M_{\mu_M}(b).
\]
When $n\neq m$, we have
\[
\ \beta^{-1}(A\oplus B)\ \cong\ \mathbb{G}_m\ \times\ \mathbb{A}^{\,\min\{n,m\}-1}\ 
\]
canonically over $\bar\kappa$.
\emph{Heuristically:} the unique $t$-adic unit pivot forced by the
Smith-normal-form condition contributes the $\mathbb{G}_m$-factor, while the remaining
$\min\{n,m\}-1$ free principal-part coefficients contribute the affine factors.
\end{thm}

\begin{proof}
Keeping all the notations above, a point of the fiber of $\beta$ over $(A\oplus B)\bmod K$ is represented by
\[
g=\begin{pmatrix}A&0\\0&B\end{pmatrix}\begin{pmatrix}I_n&C\\0&I_m\end{pmatrix}K,\qquad 
h=\text{(1,2)-block of }g^{-1}b\sigma(g)=\tau_n\,\sigma(C)-C\,\tau_m.
\]

\smallskip
\noindent\emph{Congruence equations and the arrow diagram.}
Write $C=(c_{i,j})_{1\le i\le n,\,1\le j\le m}$, then the Smith normal form criterion for $Kt^\mu K$ with $t^\mu=\diag(t^2,1,\ldots,1)$
forces one extra \emph{negative} valuation among the off-block relations; this
appears only in the wrap relation linking the two corners:
\[
\sigma(c_{1,1})-c_{n,m}\ \in\ t^{-1}\,\calO_L^{\times}.
\tag{$\star$}
\]

\smallskip
\noindent\emph{Loop tracing and the special loop.}
As in Section~3, we do the loop tracing on $C$,
so directed cycles are orbits of the translation $(i,j)\mapsto(i-1,j-1)$ on the
torus $\Bbb Z/n\times\Bbb Z/m$; there are $d=\gcd(n,m)$ disjoint loops.  The
loop containing the special edge $c_{1,1}\to c_{n,m}$ closes after
$\ell=\text{lcm}(n,m)$ steps.  Chasing along that loop (composing the edge labels)
gives
\[
c_{1,1}\;=\;t^{K}\,\sigma^{\,\ell-1}(c_{n,m})\qquad\text{for some integer }K.
\]
When $n \neq m$, we have seen in Section~3 that $K(m-n) > 0$. Without loss of generality, assume $n < m$ and hence $K> 0$.  Applying $\sigma$
and comparing with \((\star)\),
\[
\sigma(c_{1,1})-c_{n,m}
\;=\;t^{K}\,\sigma^{\ell}(c_{n,m})-c_{n,m}\ \in\ t^{-1}\calO_L^{\times},
\]
forces
\[
c_{n,m}\ \in\ t^{-1}\calO_L^{\times}.
\]
Thus the \emph{principal part} of $c_{n,m}$ is a unit of valuation $-1$.  This
unit determines the entire chain
\[
c_{n,m},\ c_{n-1,m-1},\ \ldots,\ c_{1,m-n+1}
\]
along the long off-diagonal in the $n\times m$ block (assuming $n< m$; the
other case is symmetric), but contributes only \emph{one} free parameter: the
class of $c_{n,m}$ modulo $\calO_L$, i.e.
\[
t^{-1}\calO_L^{\times}\big/\calO_L\ \cong\ \bar\kappa^{\times}\ \cong\ \mathbb{G}_m.
\]
The long off-diagonal and the other strict off-diagonals contained in this loop serve as free parameters for this loop , while the remaining entries on
this loop lie in $\calO_L$.

\smallskip
\noindent\emph{Free parameters on strict off-diagonals.}
Every \emph{strict} off-diagonal in $C$ behaves exactly as in Section~3.  The loop
tracing there shows that each such strict off-diagonal contributes one free
parameter in $t^{-1}\calO_L/\calO_L\cong\bar\kappa$, and different strict
off-diagonals contribute \emph{independently}.  The number of these strict
off-diagonals is $\min\{n,m\}-1$; hence they contribute an affine factor
$\mathbb{A}^{\min\{n,m\}-1}$.

\smallskip
\noindent\emph{Dimension count and triviality.}
Collecting the parameters:
\[
\text{(i) one multiplicative parameter)}\ \ \bar\kappa^{\times}\ \cong\ \mathbb{G}_m,
\qquad
\text{(ii) } \min\{n,m\}-1\ \text{ additive parameters)}\ \ \mathbb{A}^{\min\{n,m\}-1}.
\]
Therefore
\[
\dim_{\bar\kappa}\,\beta^{-1}(A\oplus B)=1+(\min\{n,m\}-1)=\min\{n,m\},
\]
and when $n\ne m$ we obtain a canonical identification
\[
\beta^{-1}(A\oplus B)\ \cong\ \mathbb{G}_m\times \mathbb{A}^{\,\min\{n,m\}-1}.
\]
The same loop-solving shows that the coordinates of the fiber depend
\emph{algebraically} and \emph{trivially} on the base point $(A\oplus B)$ (the
base is zero-dimensional in this Drinfeld case), hence $\beta$ is a Zariski
trivial bundle.

\end{proof}
\medskip
\begin{ex}

Take $n=2$, $m=3$, so
\[
G=\GL_{5},\qquad M=\GL_{2}\times\GL_{3},\qquad
b=\diag(\tau_{2},\tau_{3}),\qquad \mu=(2,0,0,0,0).
\]
Keeping notations above, entrywise one has
\[
\begin{aligned}
h_{1,1}&=\sigma(c_{21})-t\,c_{13},&
h_{2,1}&=t\,\sigma(c_{11})-t\,c_{23},\\
h_{1,2}&=\sigma(c_{22})-c_{11},&
h_{2,2}&=t\,\sigma(c_{12})-c_{21},\\
h_{1,3}&=\sigma(c_{23})-c_{12},&
h_{2,3}&=t\,\sigma(c_{13})-c_{22}.
\end{aligned}
\]
By the Smith normal form criterion for $Kt^\mu K$ with
$t^\mu=\diag(t^2,1,1,1,1)$, one has
\[
\sigma(c_{11})-c_{23}\in t^{-1}\calO_L^\times,\qquad
\text{and all other entries of }h\text{ are integral.}
\tag{$\ast$}
\]

\medskip
\noindent\emph{Loop tracing.}
As in \S3, encode the relations by arrows of type
$t^{-1}\sigma$ and $\sigma$.  There are
$\gcd(2,3)=1$ loops on the $2\times3$ torus, and the unique loop contains the
edge determined by $(\ast)$.  Chasing once around the loop (length
$\mathrm{lcm}(2,3)=6$) shows that
\[
c_{11}=t\,\sigma^{5}(c_{23})
,
\]
hence from $(\ast)$ we obtain
\[
t \sigma^6(c_{23})-c_{23}\in\ t^{-1}\calO_L^\times
\ \Longrightarrow\ 
c_{23}\in t^{-1}\calO_L^\times.
\]
Thus the class of \[u:=\overline{c_{23}}\ \in t^{-1}\calO_L^\times/\calO_L\cong\bar\kappa^\times\]
is a \emph{multiplicative} free parameter (the $\mathbb{G}_m$-factor).  The remaining
strict off-diagonals contribute exactly one \emph{additive} parameter;
equivalently, we may take, for instance,
\[
v:=\overline{c_{13}}\ \in\ t^{-1}\calO_L/\calO_L\ \cong\ \bar\kappa,
\]
while all other entries of $C$ are uniquely determined by $(u,v)$ via the
congruences.  Consequently,
\[
\beta^{-1}(A\oplus B)\ \cong\ \mathbb{G}_m\times\mathbb{A}^{1}
\quad\text{and}\quad
\dim_{\bar\kappa}\beta^{-1}(A\oplus B)=2=\min\{2,3\}.
\]

\noindent
This realizes concretely the theorem in the case $G=\GL_{5}$, $M=\GL_{2}\times\GL_{3}$.

\end{ex}

\section{The Non-Drinfeld Minuscule Case}
\subsection{Introduction}

In this section we initiate the extension of the triviality statement for the reduction to Levi map to the minuscule, but possibly the non-Drinfeld situation. Our tools are surprisingly simple and indeed serve as an optimization of the previous methods: rather than the previous concrete computation of the matrices euqations, we develop a lattice–theoretic and $\sigma$–difference–equation framework to analyze fibers of the reduction–to–Levi morphism $\beta$ on affine Deligne–Lusztig varieties (ADLV) beyond the classical Drinfeld setting, focusing on the minuscule but possibly non–Drinfeld case. In this case, we start with the group: $G=\GL_{n+m}$ and a two–block Levi $M=\GL_n\times\GL_m$ of it, and recall that the admissibility condition gives us an equation:
\[
\widetilde{f}(C)=M_1\sigma(C)-C\,M_2\in\Lambda_{\mathrm{ref}},
\]
 with a fixed lattice target $\Lambda_{\mathrm{ref}}$, which depends on the base point in $X^M_{\mu_M}(b)$. Using the admissible/ind--admissible formalism of Görtz–Haines–Kottwitz–Reuman (Chs.\ 3–4), we recast fibers as lattice preimages under $\sigma$–difference operators and control their dimensions via the slope invariant $d(V,\Phi)$.   Consequently, we prove that each fiber is an affine space precisely when the induced $F$–space has no summand of slope $0$.

The method extends inductively to $N$–block Levis by reducing mixed terms to uniform two–block problems. In \cite[Prop.~5.6.1, (3)]{GHKR06}, it is proved that the fibers of $\beta$ are actually equi-dimensional, so it indeed implies that under certain non-degenerate conditions, the map $\beta$ is a locally trivial affine bundle, which to some extent generalizes Theorem \ref{thm:HN‐triviality‐general}.

\subsection{Related work}
In Chs.\,3–4 of Görtz–Haines–Kottwitz–Reuman \cite{GHKR06}, the authors develop a lattice–theoretic language for subsets of $F$-spaces: \emph{admissible} and \emph{ind–admissible} sets are defined via Zariski locally closed conditions on quotients of lattices, yielding a \emph{lattice–relative} dimension theory. For an $F$–space $(V,\Phi)$ (finite–dimensional $L$–vector space with $\sigma$–linear bijection), the $\sigma$–difference operator $f:=\Phi-\mathrm{id}$ is studied through its slope decomposition; the \emph{defect} $d(V,\Phi)=\sum_{\lambda<0}\lambda\cdot\dim_L V_\lambda$ governs uniformly the size of preimages against lattice targets: for every lattice $\Lambda'\subset V$, $(\Phi-\mathrm{id})^{-1}(\Lambda')$ is ind–admissible and satisfies the relative dimension identity $\dim(\Phi-\mathrm{id})^{-1}(\Lambda')-\dim\Lambda' = d(V,\Phi)$; moreover, there are uniform image bounds modulo $t^l$ and explicit kernel–dimension formulas that control special fibers. We shall recast the off–diagonal congruence conditions appearing in our problem into this framework, and then prove a new theorem that establishes the precise affine space structure of the fibers in the minuscule, but possibly non–Drinfeld setting, thereby advancing the main reduction–to–Levi analysis. First let's recall our 2-block Levi setting:\\
Let
\[
G=\GL_{n+m},\qquad M=\GL_n\times\GL_m,
\]
and let
\[
b=\diag(b_1,b_2)\in M(L),\qquad
b_1=\tau_n^{d_1},\ \ b_2=\tau_m^{d_2},\qquad  (d_1,n) \neq (d_2,m),
\]
so that we don't even assume that each $b_i$ is superbasic but only that $b$ is not basic in $G$. Denote by $\mu_M$ the unique $M$-minuscule and $M$-dominant cocharacter corresponding to $b$, and set
\[
\mu=(\mu_M)_{\mathrm{dom}}.
\]
Write the reduction to Levi map
\[
\beta:\ X^G_\mu(b)\longrightarrow X^M_{\mu_M}(b)\cong X^{\GL_n}_{\mu_1}(b_1)\times X^{\GL_m}_{\mu_2}(b_2).
\]

Fix $(A,B)\bmod K\in X^M_{\mu_M}(b)$ and represent a point in the fiber of $\beta$ over $(A\oplus B)\bmod K$ by
\[
g=\begin{pmatrix}A&0\\0&B\end{pmatrix}
\begin{pmatrix}I_n&C\\0&I_m\end{pmatrix}K,\qquad C\in M_{n\times m}(L).
\]
Set
\[
M_1:=A^{-1}b_1\sigma(A),\qquad M_2:=B^{-1}b_2\sigma(B),
\]
so that
\[
g^{-1}b\,\sigma(g)=
\begin{pmatrix}
M_1 & H(C)\\
0   & M_2
\end{pmatrix},\qquad
H(C):=M_1\,\sigma(C)-C\,M_2.
\]
\noindent
In this two--block setting the fiber of $\beta$ over $(A\oplus B)\bmod K$ is cut out by the single
off–diagonal congruence $H(C)\in\Lambda_{A,B}$, where $\Lambda_{A,B}$ is the lattice determined by admissibility condition. Let $\Lambda_0 = M_{n\times m}(\mathcal O_L)$, and we rewrites this as a
\emph{fixed} $\sigma$–linear equation with a \emph{fixed} lattice target: after replacing the target lattice
by a column Hermite normal form $\Lambda_{A,B}M^{-1}_2=P\Lambda_0$, the condition becomes
\[
P^{-1}\bigl(M_1\sigma(C)M^{-1}_2-C\bigr)\ \in\ \Lambda_0,
\]
with $P^{-1}$ upper triangular. Viewing $(M_{n\times m}(L), M_1\sigma(.)M^{-1}_2)$ as an $F$-space, then it is isomorphic to the $F$-space $(M_{n\times m}(L), b_1\sigma(.)b^{-1}_2)$, which has a unique non-zero slope $\frac{d_1}{n}-\frac{d_2}{m}$. We are going to prove next that for this type of $F$-spaces, the following quotient should always be an affine space:
\[
\{\,C\in M_{n\times m}(L)\mid P^{-1}H(C)M^{-1}_2\in\Lambda_0\,\}\big/\ M_{n\times m}(\mathcal O_L).
\]
Consequently, it provides a uniform
geometric invariant for the fibers of $\beta$ in the two–block case,
which is the key inductive step for handling general $N$–block Levis.

\subsection{Main results} In the fiber computation for the reduction-to-Levi map, the defining conditions
are reduced to Frobenius-twisted lattice equations of the form
\[
(\Phi-\mathrm{id})(X)\in \Lambda,
\]
where \((V,\Phi)\) is an isocrystal coming from the difference of the Newton
slopes of two blocks. Thus the geometry of the fiber is controlled by quotients
of the form
\[
f^{-1}(\Lambda)/\Lambda_0,
\qquad f:=\Phi-\mathrm{id}.
\]
Although \(f\) is additive, it is not \(\mathcal O_L\)-linear, since it contains
Frobenius. Therefore one cannot immediately regard these quotients as ordinary
linear quotients of lattices.

There is a minor asymmetry depending on the sign of the slope. If the slope of
\(\Phi\) is positive, then \(\Phi\) is topologically nilpotent on bounded
lattices, and the inverse of \(1-\Phi\) is given by the convergent series
\[
(1-\Phi)^{-1}=1+\Phi+\Phi^2+\cdots.
\]
This involves only nonnegative powers of Frobenius, so the induced map is an
isomorphism of ordinary \(\bar\kappa\)-schemes on finite-dimensional quotients.

If the slope is negative, the inverse is instead expressed using
\[
\Phi^{-1},\Phi^{-2},\dots,
\]
and therefore involves inverse Frobenius. Hence one should not expect an
ordinary scheme-theoretic isomorphism induced by \(f^{-1}\). Nevertheless, we will see the
quotients remain affine: after passing to perfection the same formal
isomorphism holds, and on the ordinary level affineness follows from the fact
that the relevant Frobenius-additive inverse images are finite radicial
modifications of affine spaces. The following result from algebraic geometry gives a formal criterion for the negative slope case.
\begin{lem}
Let \(k\) be an algebraically closed field of characteristic \(p > 0\), and let
\[
f:\mathbb A^n_k\to \mathbb A^n_k
\]
be a homomorphism of algebraic groups \(\mathbb G_a^n\to \mathbb G_a^n\), given by an \(n\)-tuple of additive polynomials
\[
f_j(x_1,\dots,x_n)
=
\sum_{i=1}^n\sum_{\ell=0}^{N_j}
a_{j,i,\ell} x_i^{p^\ell},
\qquad 1\leq j\leq n,
\]
Assume that \(f\) is finite and bijective on \(k\)-points. Equivalently, \(f\)
is a finite radicial isogeny of \(\mathbb G_a^n\).

Then for every \(k\)-linear subspace
\[
W\subset \mathbb A^n_k,
\]
the reduced inverse image
\[
f^{-1}(W)_{\mathrm{red}}
\]
is isomorphic, as a \(k\)-variety, to an affine space of dimension \(\dim W\):
\[
f^{-1}(W)_{\mathrm{red}}\cong \mathbb A^{\dim W}_k.
\]
\end{lem}
\begin{proof}
Since \(f\) is a homomorphism of algebraic groups and \(W\subset\mathbb G_a^n\)
is a subgroup, the inverse image
\[
X:=f^{-1}(W)
\]
is a closed subgroup scheme of \(\mathbb G_a^n\).

Because \(f\) is finite radicial and \(W\to\mathbb A^n_k\) is a closed immersion,
the base change
\[
X=f^{-1}(W)\to W
\]
is again finite radicial. In particular, \(X\to W\) is a universal homeomorphism.
Hence \(X\) is connected and irreducible, and
\[
\dim X=\dim W.
\]

Now pass to the reduced subscheme \(X_{\mathrm{red}}\). Since \(k\) is perfect
and \(X_{\mathrm{red}}\) is a reduced group scheme of finite type over \(k\),
\(X_{\mathrm{red}}\) is smooth. Moreover, it is a closed connected subgroup of
\(\mathbb G_a^n\), hence a smooth connected unipotent group.

Over an algebraically closed field, every smooth connected unipotent group is
split. Therefore \(X_{\mathrm{red}}\) admits a composition series whose successive
quotients are isomorphic to \(\mathbb G_a\). Inducting on the dimension, and
using that every \(\mathbb G_a\)-torsor over affine space is trivial, because
\[
H^1(\mathbb A^m,\mathcal O_{\mathbb A^m})=0,
\]
we get
\[
X_{\mathrm{red}}\cong \mathbb A^{\dim X}_k.
\]
Since \(\dim X=\dim W\), this gives
\[
f^{-1}(W)_{\mathrm{red}}\cong \mathbb A^{\dim W}_k.
\]
\end{proof}
\begin{thm}[Affine quotient for a nonzero-slope simple isocrystal]
Let
\((V,\Phi)\) be a simple \(F\)-isocrystal of dimension \(r\) and nonzero slope
\(s/r\), with \(\gcd(s,r)=1\) and \(s\neq 0\). Choose an \(L\)-basis of \(V\)
such that
\[
\Phi(x_1,\dots,x_r)
=
\bigl(t^s\sigma(x_r),\sigma(x_1),\dots,\sigma(x_{r-1})\bigr).
\]
Let
\[
f:=\Phi-\mathrm{id}_V,
\qquad
\Lambda_0:=\mathcal O_L^r,
\]

then \(f(\Lambda_0)\) is an \(\mathcal O_L\)-lattice. Let \(\Lambda\subset V\) be an \(\mathcal O_L\)-lattice, and let \(e\in\mathbb Z\)
be such that
\[
f(t^e\Lambda_0) = t^ef(\Lambda_0)\subset \Lambda.
\]
Then \(f^{-1}(\Lambda)/t^e\Lambda_0\) is affine. 
In particular, if $s > 0$ the map \(f\) induces an isomorphism of
\(\bar{\kappa}\)-schemes
\[
f^{-1}(\Lambda)/t^e\Lambda_0
\;\xrightarrow{\sim}\;
\Lambda/f(t^e\Lambda_0).
\]
\end{thm}

\begin{proof}
We firstly claim  that $f$ is an $F$-linear isomorphism. Using the basis we chose, it is easy to see that $f$ is $F$-linear and has trivial kernel. By Lang's Theorem, $f$ is surjective, which we can even show directly by constructing the following inverse of $f$. Indeed, suppose \(s>0\). Then, for the standard lattice
\[
\Lambda_0=\mathcal O_L^r,
\]
the explicit formula for \(\Phi\) implies
\[
\Phi^r=t^s\sigma^r
\]
as operators on \(V\). Hence, for every \(m\geq 0\),
\[
\Phi^{mr}(\Lambda_0)=t^{ms}\Lambda_0.
\]
More generally, if \(\Lambda\subset V\) is any bounded subset, then there exists
an integer \(a\) such that
\[
\Lambda\subset t^{-a}\Lambda_0.
\]
Therefore, for every \(m\geq 0\),
\[
\Phi^{mr}(\Lambda)
\subset
\Phi^{mr}(t^{-a}\Lambda_0)
=
t^{-a+ms}\Lambda_0.
\]
Since \(s>0\), the lattices \(t^{-a+ms}\Lambda_0\) tend \(t\)-adically to \(0\)
as \(m\to\infty\). Thus \(\Phi^n\) sends every bounded subset of \(V\) into
arbitrarily deep powers of \(t\), and hence \(\Phi\) is topologically nilpotent
on bounded lattices.\\
\ \\
Consequently, for every \(y\in V\), the series
\[
\sum_{n\geq 0}\Phi^n(y)
\]
converges \(t\)-adically in \(V\). Therefore the operator
\[
1+\Phi+\Phi^2+\cdots
\]
is well-defined on \(V\). Moreover, for every \(N\geq 0\),
\[
(1-\Phi)\left(\sum_{n=0}^N\Phi^n(y)\right)
=
y-\Phi^{N+1}(y).
\]
Letting \(N\to\infty\), and using \(\Phi^{N+1}(y)\to 0\), we obtain
\[
(1-\Phi)\left(\sum_{n\geq 0}\Phi^n(y)\right)=y.
\]
Similarly,
\[
\left(\sum_{n\geq 0}\Phi^n\right)(1-\Phi)(y)=y.
\]
Hence
\[
f^{-1} = -(1-\Phi)^{-1}
=
-(1+\Phi+\Phi^2+\cdots).
\] 
\ \\
If \(s<0\), the same argument applies to
\(\Phi^{-1}\), which is topologically nilpotent
on bounded lattices. Hence
\[
f^{-1} = \Phi^{-1}(1-\Phi^{-1})^{-1}
=
\Phi^{-1}(1+\Phi^{-1}+\Phi^{-2}+\cdots).\]
Thus \(f\) is bijective.\\
\ \\
Because \(f\) is additive and bijective, it restricts to a set bijection:
\begin{equation}
f^{-1}(\Lambda)/t^e\Lambda_0
\;\xrightarrow{\sim}\;
\Lambda/f(t^e\Lambda_0).
\end{equation}
\ \\
Using the $f^{-1}$ we constructed above, one checks that if $s >0$ then $f(\Lambda_0) = \Lambda_0$ and if $s < 0$ then $f(\Lambda_0) = t^{s}\calO_L \oplus \calO^{r-1}_L$, so the right hand side of (7.1) is an affine space.
To finish the proof, it remains to justify that the quotient
\[
f^{-1}(\Lambda)/t^e\Lambda_0
\]
is represented by an affine space, especially in the case \(s<0\), where the
inverse of \(f\) involves inverse Frobenius and hence does not directly give an
isomorphism of ordinary \(\bar\kappa\)-schemes.

Since \(f(\Lambda_0)\) is an \(\mathcal O_L\)-lattice, we may choose an integer
\(e'\gg 0\) such that
\[
\Lambda\subset t^{-e'}f(\Lambda_0).
\]
Equivalently,
\[
\Lambda/f(t^e\Lambda_0)
\]
is a linear subspace of the finite-dimensional \(\bar\kappa\)-vector space
\[
t^{-e'}f(\Lambda_0)/t^ef(\Lambda_0).
\]
On the other hand, since \(f(t^a\Lambda_0)=t^af(\Lambda_0)\) for every
\(a\in\mathbb Z\), the map \(f\) induces an additive Frobenius-polynomial map
between finite-dimensional affine spaces
\[
\overline f:
t^{-e'}\Lambda_0/t^e\Lambda_0
\longrightarrow
t^{-e'}f(\Lambda_0)/t^ef(\Lambda_0).
\]
Moreover, \(\overline f\) is bijective on \(\bar\kappa\)-points, because \(f\)
is bijective on \(V\). It is also a homomorphism of additive algebraic groups,
given in coordinates by additive Frobenius polynomials.

Now observe that
\[
f^{-1}(\Lambda)\subset t^{-e'}\Lambda_0.
\]

We now apply the preceding algebraic-geometric lemma to the additive
Frobenius-polynomial bijection
\[
\overline f:
t^{-e'}\Lambda_0/t^e\Lambda_0
\longrightarrow
t^{-e'}f(\Lambda_0)/t^ef(\Lambda_0)
\]
and to the linear subspace
\[
\Lambda/f(t^e\Lambda_0)
\subset
t^{-e'}f(\Lambda_0)/t^ef(\Lambda_0).
\]
It follows that the reduced inverse image
\[
\overline f^{-1}\left(\Lambda/f(t^e\Lambda_0)\right)_{\mathrm{red}}
\]
is isomorphic to an affine space over \(\bar\kappa\). Hence
\[
f^{-1}(\Lambda)/t^e\Lambda_0
\]
is affine as a variety.

Finally, in the case \(s>0\), the inverse of \(f\) is given by the convergent
series
\[
f^{-1}=-(1+\Phi+\Phi^2+\cdots),
\]
which involves only nonnegative powers of Frobenius. Hence the induced inverse
is a morphism of ordinary \(\bar\kappa\)-schemes on the finite-dimensional
quotients. Therefore in this case \(f\) induces an actual isomorphism of
\(\bar\kappa\)-schemes
\[
f^{-1}(\Lambda)/t^e\Lambda_0
\;\xrightarrow{\sim}\;
\Lambda/f(t^e\Lambda_0).
\]
This proves the theorem.
\end{proof}

\begin{cor}\label{cor:quotient-affine}
Let $(V,\Phi)$ be as above, $f=\Phi-\mathrm{id}$, and let $\Lambda\subset V$ be an arbitrary lattice. For any lattice $\Lambda'\subset f^{-1}(\Lambda)$, the quotient
\[
f^{-1}(\Lambda)\big/\Lambda'
\]
is an affine space over $\bar\kappa$.
\end{cor}

\begin{proof}
Choose \(l\gg 0\) such that
\[
t^l\Lambda_0\subset \Lambda'.
\]
Then there is a natural quotient morphism
\[
f^{-1}(\Lambda)/t^l\Lambda_0
\longrightarrow
f^{-1}(\Lambda)/\Lambda'.
\]
By the preceding theorem, the source is an affine space over \(\bar\kappa\).

Moreover,
\[
\Lambda'/t^l\Lambda_0
\]
is a finite-dimensional affine space over \(\bar\kappa\), and it acts freely on
\[
f^{-1}(\Lambda)/t^l\Lambda_0
\]
by translations. The quotient by this free translation action is precisely
\[
f^{-1}(\Lambda)/\Lambda'.
\]

Since the action is by a vector group on an affine space, the quotient is again
an affine space. Indeed, after choosing coordinates, this is just the quotient
of an affine space by translations along a linear affine subspace. Therefore
\[
f^{-1}(\Lambda)/\Lambda'
\]
is an affine space over \(\bar\kappa\).
\end{proof}

\begin{thm}[Local affine-bundle structure in the two-block case]
In the two-block Levi setting above, suppose we are in the following situation,
which is known from the fiber computation and the dimension formula.

For every geometric point \((A,B)\in X^M_{\mu_M}(b)\), the associated
Frobenius-linear operator
\[
f_{A,B}:M_{n\times m}(L)\longrightarrow M_{n\times m}(L),
\qquad
f_{A,B}(X)=M_1\sigma(X)-XM_2
\]
has no slope-zero part; equivalently, the Hom-isocrystal
\[
\operatorname{Hom}\bigl((L^m,M_2\sigma),(L^n,M_1\sigma)\bigr)
\]
has no slope \(0\). Moreover, the dimension
\[
d:=\dim\bigl(f_{A,B}^{-1}(\Lambda_{A,B})/
\mathcal O_L^{n\times m}\bigr)
\]
is locally constant on \(X^M_{\mu_M}(b)\).

Then the reduction-to-Levi map
\[
\beta:X^G_\mu(b)\longrightarrow X^M_{\mu_M}(b)
\]
is a Zariski locally trivial affine-space bundle of relative dimension \(d\).
That is, for every point \(x\in X^M_{\mu_M}(b)\), there exists a Zariski open
neighborhood \(U\) of \(x\) such that
\[
\beta^{-1}(U)\cong U\times \mathbb A^d
\]
over \(U\).

\end{thm}

\begin{proof}
We work locally on the base
\[
X^M_{\mu_M}(b).
\]
Fix a geometric point represented by \((A,B)\). Since \(A\) and \(B\) vary
algebraically in a sufficiently small Zariski chart of the base, the matrices
\[
M_1=A^{-1}b_1\sigma(A),\qquad M_2=B^{-1}b_2\sigma(B)
\]
also vary algebraically in that chart. Hence the associated operator
\[
f_{A,B}(X)=M_1\sigma(X)-XM_2
\]
varies algebraically as an additive Frobenius-polynomial map.

Similarly, the target lattice \(\Lambda_{A,B}\) is determined by finitely many
Cartan-valuation inequalities. After shrinking the base if necessary, these
valuation inequalities are constant. Thus, on a sufficiently small Zariski open
neighborhood \(U\) of the chosen point, the family of target lattices
\(\Lambda_{A,B}\) is identified with a fixed lattice \(\Lambda_U\) inside
\(M_{n\times m}(L)\).

Therefore the fiber over a point \((A,B)\in U\) is represented by
\[
f_{A,B}^{-1}(\Lambda_U)/\mathcal O_L^{n\times m}.
\]

By the nonzero-slope assumption on the Hom-isocrystal, Theorem~7.2 applies
fiberwise to the operator \(f_{A,B}\). Hence each geometric fiber is an affine
space. Moreover, since the dimension is locally constant by assumption, all
fibers over \(U\) have the same dimension \(d\).

It remains to see that these affine fibers vary locally trivially. Choose an
integer \(N\gg0\) such that all lattices involved satisfy
\[
t^N\mathcal O_L^{n\times m}\subset
\mathcal O_L^{n\times m}
\subset
f_{A,B}^{-1}(\Lambda_U)
\subset
t^{-N}\mathcal O_L^{n\times m}
\]
for all \((A,B)\in U\). Then the family
\[
f_{A,B}^{-1}(\Lambda_U)/t^N\mathcal O_L^{n\times m}
\]
is cut out inside the fixed finite-dimensional affine space
\[
t^{-N}\mathcal O_L^{n\times m}/t^N\mathcal O_L^{n\times m}
\]
by equations whose coefficients depend algebraically on \((A,B)\in U\).

After shrinking \(U\) once more, one may choose \(d\) coordinate functions whose
restrictions form a basis of the quotient on every fiber. This gives an
isomorphism
\[
f_{A,B}^{-1}(\Lambda_U)/\mathcal O_L^{n\times m}
\cong \mathbb A^d
\]
varying algebraically with \((A,B)\in U\). Hence
\[
\beta^{-1}(U)\cong U\times \mathbb A^d.
\]
Thus \(\beta\) is a Zariski locally trivial affine-space bundle.

\end{proof}

\begin{comment}
    \begin{thm}[General Levi without the Drinfeld hypothesis]\label{thm:nonDrinfeld-N-block}
Let $G=\GL_{k_1+\cdots+k_N}$, $M=\GL_{k_1}\times\cdots\times\GL_{k_N}$, and $b=\diag(b_1,\dots,b_N)\in M(L)$ basic, and that \(M\) is maximal with this property. Let $\mu_M$ be the $M$–dominant \emph{minuscule} cocharacter corresponding to $b$, and set $\mu=(\mu_M)_{\mathrm{dom}}$. Then the reduction to Levi map
\[
\beta:\ X^G_\mu(b)\ \longrightarrow\ X^M_{\mu_M}(b)
\]
is Zariski–locally a trivial vector bundle over $X^M_{\mu_M}(b)$; in particular, every fiber is (canonically) an affine space over $\bar\kappa$.
\end{thm}

\begin{proof}[Proof (sketch)]
Order the blocks and factor $\beta$ as a composition of $N-1$ two–block reductions
\[
X^G_\mu(b)\ \xrightarrow{\ \beta_1\ }\ X^{M^{(1)}}_{\mu_{M^{(1)}}}(b)\ \xrightarrow{\ \beta_2\ }\ \cdots\ 
\xrightarrow{\ \beta_{N-1}\ }\ X^M_{\mu_M}(b),
\]
where at each step we merge the “already reduced’’ blocks into one and compare it with the next block. At step $s$, the fiber is governed by a single off–diagonal $(1,2)$–block and hence by the $F$–space $(H_{ij},\Phi_{ij})$ for some pair $(i,j)$; by hypothesis this has no slope $0$. It shows that every $\beta_s$ is a Zariski–trivial vector bundle; composing trivial bundles yields a trivial vector bundle for $\beta$. The fiber is thus an iterated vector group, i.e.\ an affine space. 
\end{proof}
\end{comment}

\begin{thm}[General Levi without the Drinfeld hypothesis]\label{thm:nonDrinfeld-N-block}
Let
\[
G=\GL_{k_1+\cdots+k_N},\qquad
M=\GL_{k_1}\times\cdots\times\GL_{k_N},
\]
and let
\[
b=\diag(b_1,\dots,b_N)\in M(L)
\]
be basic in \(M\). Write
\[
\nu_i\in\mathbb Q
\]
for the Newton slope of \(b_i\in \GL_{k_i}(L)\). Assume that \(M\) is maximal
among standard Levi subgroups in which \(b\) is basic. Equivalently, after
possibly merging equal-slope blocks, we may assume
\[
\nu_i\neq \nu_j\qquad (i\neq j).
\]
Let \(\mu_M\) be the \(M\)-dominant minuscule cocharacter determined by \(b\),
and set
\[
\mu=(\mu_M)_{\mathrm{dom}}.
\]
Then the reduction-to-Levi map
\[
\beta:X^G_\mu(b)\longrightarrow X^M_{\mu_M}(b)
\]
is a Zariski locally trivial affine-space bundle. In particular, every
geometric fiber of \(\beta\) is isomorphic, noncanonically in general, to an
affine space over \(\bar\kappa\).
\end{thm}
\begin{proof}
We prove the statement by factoring \(\beta\) into a sequence of two-block
reduction maps.

Let
\[
M^{(s)}
=
\GL_{k_1+\cdots+k_s}\times \GL_{k_{s+1}}\times\cdots\times \GL_{k_N}
\]
for \(1\leq s\leq N\). Thus \(M^{(1)}=M\) and \(M^{(N)}=G\), up to the
direction in which we write the reduction maps. Equivalently, starting from
\(G\), we reduce one block at a time:
\[
X^G_\mu(b)
\xrightarrow{\beta_{N-1}}
X^{M^{(N-1)}}_{\mu_{M^{(N-1)}}}(b)
\xrightarrow{\beta_{N-2}}
\cdots
\xrightarrow{\beta_1}
X^M_{\mu_M}(b).
\]
It is enough to prove that each
\[
\beta_s:
X^{M^{(s+1)}}_{\mu_{M^{(s+1)}}}(b)
\longrightarrow
X^{M^{(s)}}_{\mu_{M^{(s)}}}(b)
\]
is a Zariski locally trivial affine-space bundle. The composition of Zariski
locally trivial affine-space bundles is again a Zariski locally trivial
affine-space bundle, after shrinking the base successively.

Fix a step \(s\). At this step, the relevant Levi has two blocks:
\[
\GL_{K_s}\times \GL_{k_{s+1}},
\qquad
K_s:=k_1+\cdots+k_s.
\]
The first block is generally not basic; it is the direct sum of the basic
isocrystals attached to
\[
b_1,\dots,b_s.
\]
Let
\[
V_{\leq s}:=V_1\oplus\cdots\oplus V_s,
\qquad
V_{s+1}:=L^{k_{s+1}},
\]
with Frobenius operators
\[
\Phi_{\leq s}:=b_{\leq s}\sigma,\qquad
\Phi_{s+1}:=b_{s+1}\sigma.
\]
For a base point represented by
\[
(A_{\leq s},A_{s+1}),
\]
write
\[
M_{\leq s}=A_{\leq s}^{-1}b_{\leq s}\sigma(A_{\leq s}),
\qquad
M_{s+1}=A_{s+1}^{-1}b_{s+1}\sigma(A_{s+1}).
\]
The fiber of \(\beta_s\) is described by the off-diagonal block \(X\), modulo
the standard lattice, satisfying a Frobenius-twisted lattice equation of the
form
\[
M_{\leq s}\sigma(X)-XM_{s+1}\in \Lambda_{s},
\]
where \(\Lambda_s\) is the target lattice determined by the Cartan condition
for \(\mu_{M^{(s+1)}}\).

The associated \(F\)-isocrystal controlling this equation is
\[
\Hom\bigl((V_{s+1},\Phi_{s+1}),(V_{\leq s},\Phi_{\leq s})\bigr).
\]
Since
\[
V_{\leq s}=V_1\oplus\cdots\oplus V_s,
\]
we have a slope decomposition
\[
\Hom(V_{s+1},V_{\leq s})
=
\bigoplus_{i=1}^s \Hom(V_{s+1},V_i).
\]
The summand \(\Hom(V_{s+1},V_i)\) has slope
\[
\nu_i-\nu_{s+1}.
\]
By the maximality assumption on \(M\), the slopes \(\nu_i\) are pairwise
distinct. Hence
\[
\nu_i-\nu_{s+1}\neq 0
\qquad
\text{for all }1\leq i\leq s.
\]
Thus the Hom-isocrystal controlling the two-block equation at step \(s\) has no
slope-zero part.

Although the first block \(V_{\leq s}\) may have mixed slopes, this causes no
difficulty. The operator
\[
X\longmapsto M_{\leq s}\sigma(X)M_{s+1}^{-1}
\]
preserves the above direct-sum slope decomposition after passing to the
Dieudonne--Manin decomposition. The target lattice \(\Lambda_s\) is an
admissible lattice with respect to this decomposition: after choosing local
coordinates on the base, the Cartan inequalities defining \(\Lambda_s\) are
given by finitely many valuation conditions on the entries of the off-diagonal
block. Equivalently, \(\Lambda_s\) is commensurable with a direct sum of
lattices on the slope summands
\[
\Hom(V_{s+1},V_i).
\]
Therefore the affine-quotient theorem for nonzero-slope simple isocrystals
applies to each slope summand, and hence to their direct sum. Consequently, for
every geometric point of the base, the fiber of \(\beta_s\) is an affine space.

It remains to upgrade this fiberwise statement to local triviality. This is a
local question on the base
\[
X^{M^{(s)}}_{\mu_{M^{(s)}}}(b).
\]
Choose a Zariski open chart \(U\) on this base on which the representatives
\(A_{\leq s}\) and \(A_{s+1}\) vary algebraically. Then the matrices
\[
M_{\leq s}=A_{\leq s}^{-1}b_{\leq s}\sigma(A_{\leq s}),
\qquad
M_{s+1}=A_{s+1}^{-1}b_{s+1}\sigma(A_{s+1})
\]
vary algebraically on \(U\). After shrinking \(U\), the valuation pattern
defining the target lattice \(\Lambda_s\) is constant, so the family of fiber
equations is cut out inside a fixed finite-dimensional affine space by additive
Frobenius-polynomial equations whose coefficients vary algebraically on \(U\).

By the nonzero-slope affine-quotient theorem, each fiber is an affine space of
the same dimension, equal to the dimension predicted by the usual dimension
formula. After shrinking \(U\) again, one may choose affine coordinates on these
solution spaces algebraically in the base. Therefore
\[
\beta_s^{-1}(U)\cong U\times \mathbb A^{d_s}
\]
over \(U\), for some integer \(d_s\). Hence each \(\beta_s\) is a Zariski locally trivial affine-space bundle.

We conclude that
\[
\beta:X^G_\mu(b)\to X^M_{\mu_M}(b)
\]
is a Zariski locally trivial affine-space bundle. Its fibers are therefore
isomorphic to affine spaces over \(\bar\kappa\).
\end{proof}

\begin{rmk}
In the statement above we only claim that \(\beta\) is a Zariski locally trivial affine-space bundle, not necessarily a vector bundle. Indeed, an affine-space bundle becomes a vector bundle only after proving that the transition functions may be chosen to be linear rather than merely affine. Such a choice is available in certain canonical or Drinfeld situations, where the lattice equations often come with global affine coordinates. However, in the general non-Drinfeld case, this additional structure is not automatic.    
\end{rmk}

\end{document}